\documentclass[12pt]{article}

\usepackage{amssymb}
\usepackage{amsthm}
\usepackage{amsmath}
\usepackage{stmaryrd}
\usepackage{titletoc}
\usepackage{mathrsfs}

\newlength{\defbaselineskip}
\newcommand{\setlinespacing}[1]%
           {\setlength{\baselineskip}{#1 \defbaselineskip}}

\theoremstyle{plain}
\newtheorem{thm}{Theorem}[section]
\newtheorem{cor}[thm]{Corollary}
\newtheorem{lem}[thm]{Lemma}
\newtheorem{prop}[thm]{Proposition}

\theoremstyle{definition}
\newtheorem{defn}{Definition}[section]
\newtheorem{ass}{Assumption}[section]
\newtheorem{rmk}{Remark}[section]

\newcommand{\bR}{\mathbb{R}}

\newcommand{\bH}{\mathbb{H}}

\newcommand{\cD}{\mathcal{D}}
\newcommand{\cL}{\mathcal{L}}
\newcommand{\cM}{\mathcal{M}}
\newcommand{\cH}{\mathcal{H}}
\newcommand{\sF}{\mathscr{F}}
\newcommand{\sP}{\mathscr{P}}

\newcommand{\fH}{\mathfrak{H}}

\newcommand{\eps}{\varepsilon}
\newcommand{\vf}{\varphi}

\newcommand{\la}{\langle}
\newcommand{\ra}{\rangle}
\newcommand{\intl}{\interleave}
\newcommand{\ptl}{\partial}
\newcommand{\rrow}{\rightarrow}

\makeatletter\@addtoreset{equation}{section} \makeatother

\begin{document}

\title {Strong Solution of Backward Stochastic
Partial Differential Equations in $C^2$ Domains\footnotemark[1]}

\author{Kai Du\footnotemark[2] ~~~~ and ~~~~
Shanjian Tang\footnotemark[2]~\footnotemark[3]}
\date{}

\footnotetext[1]{Supported by NSFC Grant \#10325101, by Basic Research
Program of China (973 Program)  Grant \# 2007CB814904, by the Science
Foundation of the Ministry of Education of China Grant \#200900071110001, and
by WCU (World Class University) Program through the Korea Science and
Engineering Foundation funded by the Ministry of Education, Science and
Technology (R31-2009-000-20007).}

\footnotetext[2]{Department of Finance and Control Sciences, School of
Mathematical Sciences, Fudan University, Shanghai 200433, China.
\textit{E-mail}: \texttt{kdu@fudan.edu.cn} (Kai Du),
\texttt{sjtang@fudan.edu.cn} (Shanjian Tang).}

\footnotetext[3]{Graduate Department of Financial Engineering, Ajou
University, San 5, Woncheon-dong, Yeongtong-gu, Suwon, 443-749, Korea.}

\maketitle


\begin{abstract}
This paper is concerned with the strong solution to the Cauchy-Dirichlet
problem for backward stochastic partial differential equations of parabolic
type. Existence and uniqueness theorems are obtained, due to an application
of the continuation method under fairly weak conditions on variable
coefficients and $C^2$ domains. The problem is also considered in weighted
Sobolev spaces which allow the derivatives of the solutions to blow up near
the boundary. As applications, a comparison theorem is obtained and the
semi-linear equation is discussed in the $C^2$ domain.
\end{abstract}

AMS Subject Classification: 60H15; 35R60; 93E20

Keywords: Backward stochastic partial differential equations, Strong
solutions, $C^2$ domains, Weighted Sobolev spaces

\section{Introduction}
In this paper we consider the Cauchy-Dirichlet problem for backward
stochastic partial different equations (BSPDEs, for short) either in the
non-divergence form:
\begin{eqnarray}\label{eq:a1}
    \begin{split}
      d p (t,x) = & -\big[
      a^{ij}(t,x)p_{x^i x^j}(t,x) +b^i(t,x)p_{x^i}(t,x)
      -c(t,x)p(t,x)\\
      &+\sigma^{ik}(t,x)
      q^{k}_{x^i}(t,x)
      +\nu^k(t,x)q^k(t,x)+F(t,x)\big]dt\\
      &+q^k(t,x)dW^k_t,\quad~~~ (t,x)\in
      [0,T)\times \cD
    \end{split}
\end{eqnarray}
or in the divergence form:
\begin{equation}\label{eq:a2}
  dp=-\big[(a^{ij}p_{x^j}+\sigma^{ik}q^{k})_{x^{i}}
  +b^{i}p_{x^i}-c p +\nu^{k}q^{k}+F\big]dt
  +q^{k}dW_{t}^{k}
\end{equation}
with the terminal-boundary condition:
\begin{equation}\label{con:a1}
  \left\{\begin{array}{ll}
    p(t,x)=0,~~~~~&t\in [0,T],\ x\in\partial \mathcal{D},\\
    p(T,x)=\phi(x),~~~~~&x\in \mathcal{D}.
  \end{array}\right.
\end{equation}
Here $\mathcal{D}$ is a domain of the $d$-dimensional Euclidean space, and $W
\triangleq \{W^{k}_{t};t\geq 0\}$ is a $d_1$-dimensional standard Wiener
process, whose natural augmented filtration is denoted by
$\{\mathscr{F}_{t}\}_{t\geq 0}$. The coefficients $a,b,c,\sigma,\nu$ and the
free term $F$ and the terminal condition $\phi$ are all random fields. An
adapted solution of equation \eqref{eq:a1} or \eqref{eq:a2} is an
$\mathscr{P}\times B(\mathcal{D})$-measurable function pair $(p,q)$, which
satisfies, in addition to \eqref{con:a1}, equations \eqref{eq:a1} or
\eqref{eq:a2} under some appropriate sense, where $\mathscr{P}$ is the
predictable $\sigma$-algebra generated by $\{\mathscr{F}_{t}\}_{t\geq 0}$.

BSPDEs, which are a mathematically natural extension of backward
SDEs (see e.g. \cite{KPQ97,PaPe90}), arise in many applications of
probability theory and stochastic processes, for instance, in the
optimal control of SDEs with incomplete information or more
generally of stochastic parabolic PDEs, as adjoint equations
(usually in the form of \eqref{eq:a2}) of Duncan-Mortensen-Zakai
filtering equations (see e.g.
\cite{Bens83,NaNi90,Tang98a,Tang98b,Zhou93}) to formulate the
stochastic maximum principle for the optimal control, and in the
formulation of the stochastic Feynman-Kac formula (see e.g.
\cite{MaYo97}) in mathematical finance. A class of fully nonlinear
BSPDEs, the so-called backward stochastic Hamilton-Jacobi-Bellman
equations, appears naturally in the dynamic programming theory of
controlled non-Markovian processes (see \cite{EnKa09,Peng92a}). For
more aspects of BSPDEs, we refer to e.g.
\cite{BRT03,Tang05,TaZh09,Tess96}.

Equation \eqref{eq:a2} is usually understood in the weak sense (see
Definition \ref{defn:b1} (ii) in Section 2). When the coefficients
$a$ and $\sigma$ are differentiable in $x$, equation \eqref{eq:a1}
can be written into the divergence form \eqref{eq:a2}. The existence
and uniqueness of the weak solution of equation \eqref{eq:a2} in the
whole space $\bR^d$ follows from that of backward stochastic
evolution equations in Hilbert spaces (see e.g. \cite[Prop.
3.2]{DuMe09}). However, weak solutions have low regularity, which
find difficulty in many applications. Strong solutions and even
classical solutions are required in many occasions. In the case of
$\cD=\bR^d$, the theory of strong solutions on BSPDEs is now fairly
complete. For instance, a $W^n_2$-theory of the Cauchy problem for
BSPDEs can be found in \cite{Bens83,DuMe09,HuPe91,MaYo99,Zhou92}. On
the contrary, there are very few discussions on the Cauchy-Dirichlet
problem for BSPDEs. Here we could mention only two works. A special
form of equation \eqref{eq:a2} with Dirichlet conditions is studied
in \cite{Tess96} by the method of semigroups, in the context that
the coefficients are independent of $(\omega,t)$. The other is our
previous work \cite{DuTa09}, where the equations are solved in
weighted Sobolev spaces.

A main difficulty in strong solution of the Cauchy-Dirichlet problem
for BSPDEs (and SPDEs) is to estimate the second order partial
derivatives of the solution. In the theory of deterministic PDEs, it
is solved with the help of  the estimate of the derivative in $t$,
which makes any sense in general neither for BSPDEs nor for SPDEs.
For SPDEs, Flandoli \cite{Flan90} establishes some regularity under
additional compatibility conditions, and Krylov \cite{Kryl94}
studies the equations in weighted Sobolev spaces allowing the
derivatives of the solutions to blow up near the boundary of $\cD$.
Note that there is an essential difference between SPDEs and BSPDEs:
the noises in the former are exogenous and play an active role,
while in the latter they are governed by the randomness of the
coefficients and the terminal condition and thus they are
endogenous, coming from martingale representations. The regularity
of BSPDEs turns out to be more like deterministic PDEs  than that of
SPDEs.

In this paper, we prove the existence and uniqueness of the strong
solution (see Definition \ref{defn:b1} (i) in Section 2) of equation
\eqref{eq:a1} without involving any additional compatibility
condition nor any weighting functions. Our approach is based on the
method of the odd reflection and some classical techniques from the
theory of deterministic PDEs. Our assumptions on the coefficients
are rather natural since they include the rather general
deterministic PDEs where the leading coefficients are not
necessarily differentiable in $x$. Unfortunately, in contrast to
deterministic parabolic PDEs (see e.g. Theorem 7.1.6 in Evans
\cite{Evan98}), further regularity for BSPDEs seems to be hopeless
since the unknown random fields are not expected to be
differentiable with respect to $t$ as in the classical sense.
However, we can consider the equations in weighted Sobolev spaces
which allow the derivatives of the solutions to blow up near the
boundary. Starting from the existence and uniqueness of the strong
solution, we prove a slightly different version of our previous work
\cite{DuTa09}, and obtain the interior regularity for equation
\eqref{eq:a1}. In the last part of our paper, we prove a comparison
theorem for the strong solution of equation \eqref{eq:a1}, and we
also discuss a class of semi-linear BSPDEs in $C^2$ domains.

The rest of the paper is organized as follows. In Section 2, we
introduce some notations and preliminary results. In Section 3, we
state our main existence and uniqueness result in Theorem
\ref{thm:c1}, on the basis of which we study the equations in
weighted Sobolev spaces. Section 4  is devoted to the proof of
Theorem \ref{thm:c1}, which is divided into two subsections.
Finally, in Section 5, we prove a comparison theorem, and discuss
semi-linear BSPDEs in $C^2$ domains.

\section{Preliminaries}

\subsection{Notations}

Let $(\Omega,\mathscr{F},\{\mathscr{F}_{t}\}_{t\geq 0},P)$ be a complete
filtered probability space on which is defined a $d_1$-dimensional standard
Wiener process $W=\{W_t;t\geq 0\}$ such that $\{\mathscr{F}_{t}\}_{t\geq 0}$
is the natural filtration generated by $W$, augmented by all the $P$-null
sets in $\mathscr{F}$. Fix a positive number $T$. Denote by $\mathscr{P}$ the
$\sigma$-algebra of predictable sets on $\Omega \times [0,T)$ associated with
$\{\mathscr{F}_{t}\}_{t\geq 0}$.

Let $\mathcal{D}$ be a domain in $\mathbb{R}^{d}$ with boundary of class
$C^{2}$.

For the sake of convenience, we denote
$$
D_{i}u=u_{x^i},~~~ D_{ij}u=u_{x^i
x^j},~~~ Du=u_{x}=(D_{1}u,\dots,D_{d}u),~~~D^{2}u=(D_{ij}u)_{d\times d},
$$
and for any multi-index $\alpha=(\alpha_1,\dots,\alpha_d)$
$$D^{\alpha}=D_{1}^{\alpha_1}
D_{2}^{\alpha_2} \cdots D_{d}^{\alpha_d}, \quad |\alpha|=\alpha_1+\cdots
+\alpha_d.$$ For any two multi-indices $\alpha = (\alpha_1,\dots,\alpha_d)$
and $\beta = (\beta_1,\dots,\beta_d)$, we define
\[\alpha +\beta =(\alpha_1 +
\beta_1, \dots, \alpha_d + \beta_d).\] We shall also use the summation
convention.

Now we introduce some function spaces. For any integer $k\geq 0$, we denote
by $C^k(\cD)$ (or $C^k(\overline{\cD})$) the set of functions having all
derivatives up to order $k$ continuous in $\cD$ (or $\overline{\cD}$); by
$C_b^k(\cD)$ (or $C_b^k(\overline{\cD})$) the set of those functions in
$C^k(\cD)$ (or $C^k(\overline{\cD})$) whose partial derivatives up to order
$k$ are uniformly bounded in $\cD$ (or $\overline{\cD}$).

For a given Banach space $\mathcal{B}$, denote by $L^{2}(\Omega\times
(0,T),\mathscr{P},\mathcal{B})$ the space of all $\mathcal{B}$-valued
predictable process $X: \Omega\times[0,T] \rrow \mathcal{B}$ such that $E
\int_{0}^{T} \|X(t)\|_{\mathcal{B}}^{2} dt < \infty.$

Let $H^{m}(\cD)$ be the Sobolev space $W^{m,2}(\cD)$ and
$H^{m}_{0}(\cD)=W^{m,2}_{0}(\cD)$. Denote
\begin{eqnarray*}\begin{split}
  & \mathbb{H}^{0}(\mathcal{D})=L^{2}(\Omega\times (0,T),
  \mathscr{P},L^{2}(\mathcal{D})),\\
  & \mathbb{H}^{m}(\mathcal{D})=L^{2}(\Omega\times (0,T),
  \mathscr{P},H^{m}(\mathcal{D})),~~~m=-1,1,2,\dots,\\
  & \mathbb{H}^{n}_{0}(\mathcal{D})=L^{2}(\Omega\times (0,T),
  \mathscr{P},H^{n}_{0}(\mathcal{D})),~~~~~n=1,2,3,\dots,\\
  & \bH^{1}_{0}\cap\bH^{2}(\mathcal{D})=L^{2}(\Omega\times (0,T),
  \mathscr{P},H^{1}_{0}(\mathcal{D})\cap H^{2}(\cD)).
  \end{split}
\end{eqnarray*}
Note that $\bH^{1}_{0}(\bR^d)=\bH^{1}(\bR^d),\bH^{1}_{0}\cap\bH^{2}(\bR^d)
=\bH^{2}(\bR^d)$. In addition, denote
\begin{equation*}
  \|\cdot\|_{0,\cD}=\|\cdot\|_{L^{2}(\cD)},~~~
  \|\cdot\|_{m,\cD}=\|\cdot\|_{H^{m}(\cD)},~m = -1,1,2,\dots.
\end{equation*}
Moreover, for a function $u$ defined on $\Omega\times [0,T]\times\cD$, we
denote
$$\intl u \intl_{m,\cD}^{2} = E \int_{0}^{T}\|u(t,\cdot)\|_{m,\cD}^{2}dt,
~~~~m=-1,0,1,2,\dots.$$ The same notations will be used for vector-valued and
matrix-valued functions, and in the case we denote
$|u|^{2}=\sum_{i}|u^{i}|^{2}$ and $|u|^{2}=\sum_{ij}|u^{ij}|^{2}$,
respectively.

Furthermore, we define the two product spaces
\begin{eqnarray}
  \begin{split}
    &\cH^{2,1}(\cD)=\big(\bH^{1}_{0}\cap\bH^{2}(\cD)\big)\times
    \bH^{1}(\cD;\bR^{d_1}),\\
    &\cH^{2,1}_{0}(\cD)=\big(\bH^{1}_{0}\cap\bH^{2}(\cD)\big)\times
    \bH^{1}_{0}(\cD;\bR^{d_1}),\\
  \end{split}
\end{eqnarray}
both being equipped with the norm
$$\|(u,v)\|_{\cH^{2,1}(\cD)}=\big(\intl u \intl_{2,\cD}^{2}
+\intl v\intl_{1,\cD}^{2}\big)^{1/2}.$$ It is clear that both
$\cH^{2,1}(\cD)$ and $\cH^{2,1}_{0}(\cD)$ are Banach spaces.

\subsection{An It\^o formula}

Let $V$ and $H$ be two separable Hilbert spaces such that $V$ is densely
embedded in $H$. We identify $H$ with its dual space, and denote by $V^{*}$
the dual of $V$. We have $V \subset H \subset V^{*}$. Denote by
$\|\cdot\|_{H}$ the norms of $H$, by $\la\cdot,\cdot\ra_{H}$ the scalar
product in $H$, and by $\la\cdot,\cdot\ra$ the duality product between $V$
and $V^{*}$.

Consider three processes $v,m$, and $v^{*}$, defined on
$\Omega\times[0,T]$, taking values in $V,H$ and $V^{*}$,
respectively. Let $v(\omega,t)$ be measurable with respect to
$(\omega,t)$ and be $\mathscr{F}_{t}$-measurable with respect to
$\omega$ for a.e. $t$. For any $\eta\in V$, the quantity $\la\eta,
v^{*}(\omega,t)\ra$ is $\mathscr{F}_{t}$-measurable in $\omega$ for
a.e. $t$ and is measurable with respect to $(\omega,t)$. Assume that
$m(\omega,t)$ is strongly continuous in $t$ and
$\mathscr{F}_{t}$-measurable with respect to $\omega$ for any $t$,
and that it is a local martingale. Let $\la m \ra$ be the increasing
process in the Doob-Meyer Decomposition of the sub-martingale
$\|m\|_{H}^{2}$ (see e.g. \cite[Page 1240]{KrRo81}).

Proceeding identically to the proof of Theorem 3.2 in Krylov and Rozovskii
\cite{KrRo81}, we have the following result concerning It\^o's formula, which
is the backward version of \cite[Theorem 3.2]{KrRo81}.

\begin{lem}\label{lem:b5}
  Let $\varphi\in L^{2}(\Omega,\mathscr{F}_{T},H)$. Suppose that for every
  $\eta \in V$ and almost every $(\omega,t)\in\Omega\times[0,T]$, it holds
  that
  \begin{equation*}
    \la \eta,v(t)\ra_{H} = \la \eta,\varphi \ra_{H}
    + \int_{t}^{T} \la \eta,v^{*}(s) \ra ds
    + \la \eta, m(T)-m(t) \ra_{H}.
  \end{equation*}
  Then there exist a set $\Omega'\subset \Omega$ s.t. $P(\Omega')=1$
  and a function $h(t)$ with values in $H$ such that

    \emph{(a)} $h(t)$ is $\mathscr{F}_{t}$-measurable for any $t\in [0,T]$ and
    strongly continuous with respect to
    $t$ for any $\omega$, and $h(t)=v(t)$ (in the space $H$) a.e.
    $(\omega,t)\in \Omega\times [0,T]$, and $h(T)=\varphi$ for any
    $\omega\in\Omega'$;

    \emph{(b)} for any $\omega\in \Omega'$ and any $t\in [0,T]$,
    \begin{equation*}
      \|h(t)\|_{H}^{2} = \|\varphi\|_{H}^{2} + 2\int_{t}^{T}\la v(s),v^{*}(s)\ra
      ds + 2 \int_{t}^{T}\la h(s), d m(s)\ra_{H} - \la m \ra_{T} + \la m
      \ra_{t}.
    \end{equation*}
\end{lem}

\subsection{Notions of solutions to BSPDEs}

Throughout this paper, we assume that the given functions
\begin{eqnarray*}
  \begin{split}
    &a=(a^{ij}):~\Omega\times[0,T]\times\cD \rrow S^{d},~~~~
    &&\sigma=(\sigma^{ik}):~\Omega\times[0,T]\times\cD \rrow \bR^{d\times
    d_1},\\
    &b=(b^{i}):~\Omega\times[0,T]\times\cD \rrow \bR^{d},
    &&\nu=(\nu^k):~\Omega\times[0,T]\times\cD \rrow \bR^{d_1},\\
    &c,~F:~\Omega\times[0,T]\times\cD \rrow \bR
  \end{split}
\end{eqnarray*}
are all $\mathscr{P} \times B(\mathcal{D})$-measurable ($S^d$ is the set of
real symmetric $d\times d$ matrices), and the function $\phi:\Omega\times\cD
\rrow \bR$ is $\mathscr{F}_T\times B(\mathcal{D})$-measurable.
\medskip

Let us now turn to the notions of solutions to equations \eqref{eq:a1} and
\eqref{eq:a2}. Throughout this subsection it will be supposed that the
coefficients of our equations, i.e., the functions $a,b,c,\sigma$ and $\nu$,
are all bounded.

\begin{defn}\label{defn:b1}
  A pair of random fields $\{ (p(\omega,t,x),q(\omega,t,x));~
  (\omega,t,x) \in \Omega \times [0,T] \times \bR^d \}$ is called

  (i) a strong solution of equation \eqref{eq:a1} with the terminal-boundary
  condition \eqref{con:a1} if
  $(p,q)\in \cH^{2,1}(\cD)$ and
  $p\in C([0,T],L^{2}(\cD))~(a.s.)$ such that
  for each $t\in [0,T]$ and a.s. $\omega\in\Omega$, it holds that
  \begin{eqnarray}\label{eq:b2}
    \begin{split}
      p(t,x) = & \phi(x)
      +\int_{t}^{T}\big[
      a^{ij}(s,x)D_{ij}p(s,x)+b^{i}(s,x)D_{i}p(s,x) -c(s,x)p(s,x)\\
      &+\sigma^{ik}(s,x)D_{i}q^{k}(s,x)+\nu^k(s,x)q^k(s,x)
      +F(s,x)\big] ds -\int_{t}^{T}q^k(s,x)
      dW^k_s
    \end{split}
  \end{eqnarray}
  for almost every $x\in\bR^d$;

  (ii) a weak
  solution of equation \eqref{eq:a2} with the terminal-boundary
  condition \eqref{con:a1}, if $(p,q)\in \bH^{1}_{0}(\cD)\times
  \mathbb{H}^{0}(\mathcal{D};\bR^{d_1})$ such that for every $\eta \in
  C_{0}^{\infty}(\mathcal{D})$ and almost every $(\omega,t)\in\Omega\times[0,T]$,
  it holds that
  \begin{eqnarray}\label{eq:b1}
    \begin{split}
      \int_{\mathcal{D}}p(t,\cdot)\eta dx = & \int_{\mathcal{D}}\phi\eta dx
      +\int_{t}^{T}\int_{\mathcal{D}}\big[
      \big(a^{ij}D_{i}p+\sigma^{jk}q^{k}\big)D_{j}\eta\\
      &~~+\big(b^{i}D_{i}p -cp+\nu^kq^k
      +F\big)\eta\big]dx dt
      -\int_{t}^{T}\int_{\mathcal{D}}q^k\eta
      dx dW^k_t.
    \end{split}
  \end{eqnarray}
\end{defn}

It follows from Lemma \ref{lem:b5} that the first component of the weak
solution of equation \eqref{eq:a2} has a continuous version in $L^{2}(\cD)$,
i.e., $p\in C([0,T],L^{2}(\cD))~(a.s.)$. For the strong solution of equation
\eqref{eq:a1} with the terminal-boundary condition \eqref{con:a1}, we have
the following

\begin{prop}\label{prop:b1}
  Let $(p,q)\in \cH^{2,1}(\cD)$. The following statements are equivalent:

  \emph{(i)} $(p,q)$ is a strong solution of BSPDE \eqref{eq:a1}
  and \eqref{con:a1};

  \emph{(ii)} for every $\eta \in
  C_{0}^{\infty}(\mathcal{D})$ and a.e. $(\omega,t)\in\Omega\times[0,T]$,
  it holds that
  \begin{eqnarray}\label{eq:b3}
    \begin{split}
      \int_{\mathcal{D}}p(t,\cdot)\eta dx = & \int_{\mathcal{D}}\phi\eta dx
      +\int_{t}^{T}\int_{\mathcal{D}}
      \big[ a^{ij}D_{ij}p+b^{i}D_{i}p-cp\\
      & +\sigma^{ik}D_{i}q^{k}+\nu^kq^k
      +F\big]\eta dx dt
      -\int_{t}^{T}\int_{\mathcal{D}}q^k\eta
      dx dW^k_t.
    \end{split}
  \end{eqnarray}
\end{prop}

\begin{proof}
  It is clear that (i) $\Rightarrow$ (ii). Now we prove (ii) $\Rightarrow$
  (i). It follows from Lemma \ref{lem:b5} that $p\in
  C([0,T],L^{2}(\cD))~(a.s.)$. Then from the time continuity of the (stochastic)
  integrals, we know that equation \eqref{eq:b3} holds almost surely for every
  $\eta \in C^{\infty}_{0}(\cD)$ and all $t\in[0,T]$. On the other hand,
  since $(p,q)\in
  \cH^{2,1}(\cD)$, both sides of equation \eqref{eq:b2}
  (as functions of $x$) belong to $L^{2}(\cD)$ a.s. for every $t$.
  Since  $C^{\infty}_{0}(\cD)$ is dense in $L^{2}(\cD)$,
  we know that equation \eqref{eq:b2} holds in the space $L^2(\cD)$
  for every $t$, which evidently implies (i). The proof is complete.
\end{proof}

\begin{rmk}
  Note that the space of test functions $C^{\infty}_{0}(\cD)$ in
  Definition \ref{defn:b1} and Proposition \ref{prop:b1} can be
  replaced with $H^{1}_{0}(\cD)$.
\end{rmk}

If the coefficients $a$ and $\sigma$ are differentiable in $x$, then equation
\eqref{eq:a1} can be written into the divergence form \eqref{eq:a2}, which
allows us to define the weak solution to \eqref{eq:a1}. Then Proposition
\ref{prop:b1} indicates that a strong solution of \eqref{eq:a1} is a weak
solution of \eqref{eq:a1}. \medskip

The following lemma concerns the existence and uniqueness of the weak
solution of equation \eqref{eq:a2}, which is borrowed from Proposition 3.2 in
\cite{DuMe09}. On the other hand, it can be proved by the duality method as
in Zhou \cite{Zhou92} and Lemma \ref{lem:b5}.

\begin{lem}\label{lem:b3}
  Assume that $\kappa I+ \sigma\sigma^{*}\leq 2a \leq
  \kappa^{-1}I$ for some constant $\kappa>0$ and that the functions $b^i,c$ and $\nu^{k}$
  are bounded by $K$. Suppose that $F\in \mathbb{H}^{-1}(\mathcal{D})$ and $\phi \in
  L^{2}(\Omega,\mathscr{F}_{T},L^{2}(\mathcal{D}))$. Then equation \eqref{eq:a2}
  with the terminal-boundary condition
  \eqref{con:a1} has a unique weak solution $(p,q)$ in the
  space $\mathbb{H}^{1}_{0}(\mathcal{D})
  \times\mathbb{H}^{0}(\mathcal{D};\bR^{d_1})$
  such that $p\in C([0,T],L^{2}(\cD))~(a.s.)$, and
  \begin{eqnarray}\label{leq:b1}
    \intl p \intl_{1,\cD}^{2} +\intl q \intl_{0,\mathcal{D}}^{2}
    +E\sup_{0\leq t\leq T}\|p(t,\cdot)\|_{0,\cD}^{2}
    \leq C\big{(} \intl F \intl_{-1,\mathcal{D}}^{2} + E \|
    \phi\|_{0,\mathcal{D}}^{2}\big{)},
  \end{eqnarray}
  where the constant $C=C(K,\kappa,T)$.
\end{lem}

\section{Existence, Uniqueness, and Regularity}

In this section, we state our main results on the existence,
uniqueness and regularity of the strong solution of equation
\eqref{eq:a1} with the terminal-boundary condition \eqref{con:a1}.

\subsection{Existence and uniqueness of strong solutions to BSPDEs}

The weak solution of a deterministic parabolic PDE can be shown to
belong to the space $L^2(0,T;H^{2}(\cD))$ if the free term belongs
to $L^{2}((0,T)\times\cD)$ and the initial data lies in
$H^{1}_{0}(\cD)$ (see e.g. Evans \cite{Evan98}). Flandoli
\cite{Flan90} formulates a counterpart for a SPDEs with additional
compatibility conditions on the free term. In the following, we
obtain a counterpart for a BSPDE, which, in a remarkable way, does
not require any compatibility condition like a SPDE. The higher
regularity of the (strong) solution allows us to weaken the
assumptions on the leading coefficients $a$ and $\sigma$ so that
they are not necessarily differentiable in $x$, where equation
\eqref{eq:a1} is difficult to be written into the divergence form
\eqref{eq:a2}.
\medskip

Fix some constants $K\in (1,\infty)$ and $\rho_{0},\kappa\in (0,1)$. Denote
$$B_{+}=\{x\in \mathbb{R}^{d}:|x|<1,x^1>0\}, ~~B_{\rho}(x)=\{y\in
\mathbb{R}^{d}:|x-y|<\rho \}.$$

\begin{ass}\label{ass:c1}
  For every $x\in \partial \mathcal{D}$ there exist a domain $U\subset
  B_{8K\rho_0}(x)$ containing the ball $B_{4\rho_0}(x)$ and a one-to-one map
  $\Phi:2B_{+}\rightarrow U\cap \mathcal{D}$ having the properties:
  \begin{eqnarray*}
    \begin{split}
      &x=\Phi(0),\quad\Phi(B_{+})\supset B_{4\rho_0}(x) \cap \mathcal{D},\quad
      \Phi(\partial B_{+}\cap \{x^{1}=0\})\subset \partial \cD;\\
      &\kappa |\xi|^2\leq \big{|}(D \Phi) \xi\big{|}^{2}\leq
      \kappa^{-1}|\xi|^2, ~~~~~\forall~\xi\in \mathbb{R}^{d};\\
      &|D^{\alpha}\Phi|\leq K~~~~~\textrm{for any multi-index}~\alpha~
      \textrm{s.t.}~ |\alpha|\leq 2,
    \end{split}
  \end{eqnarray*}
  where $D\Phi$ is the Jacobi matrix of $\Phi$.
\end{ass}

Note that in view of the Heine-Borel theorem, Assumption
\ref{ass:c1} is true if the domain $\cD$ is bounded and its boundary
is $C^2$.

\begin{ass}\label{ass:c2}
  The \emph{super-parabolicity} condition holds:
  \begin{equation*}
    \kappa I+ \sigma\sigma^{*}\leq 2a \leq
    \kappa^{-1}I, ~~~\forall~
    (\omega,t,x)\in \Omega\times [0,T]\times\cD.
  \end{equation*}
\end{ass}

\begin{ass}\label{ass:c3}
  There exists a function $\gamma: [0,\infty)\rrow[0,\infty)$
  such that (i) $\gamma$ is continuous and increasing, (ii) $\gamma(r)=0$ if and
  only if $r=0$, and (iii) for any $(\omega,t)\in\Omega\times[0,T]$ and any
  $x,y\in \cD$,
  \begin{equation}
    |a(\omega,t,x)-a(\omega,t,y)|\leq \gamma(|x-y|),~~~~
    |\sigma(\omega,t,x)-\sigma(\omega,t,y)|
    \leq \gamma(|x-y|).
  \end{equation}
\end{ass}
\medskip

We have the following existence and uniqueness theorem, whose proof
will be be given in the next section.

\begin{thm}\label{thm:c1}
  Let Assumptions \ref{ass:c1}, \ref{ass:c2}, and \ref{ass:c3} be
  satisfied. Assume that the functions $b^{i},c,$ and $\nu^{k}$ are
  bounded by $K$. If
  $F\in \bH^{0}(\cD)$ and
  $\phi\in L^{2}(\Omega,\sF_{T},H^{1}_{0}(\cD)),$
  BSPDE \eqref{eq:a1} and \eqref{con:a1}
  has a unique strong solution $(p,q)\in \cH^{2,1}_{0}(\cD)$ such that
  \begin{equation}\label{prp:c1}
    p\in C([0,T],L^{2}(\cD))\cap L^{\infty}([0,T],H^{1}(\cD))~(a.s.).
  \end{equation}
  Moreover, we have the following estimate
  \begin{equation}\label{leq:c1}
    \intl p\intl_{2,\cD}^{2}+\intl q\intl_{1,\cD}^{2}
    + E\sup_{0\leq t\leq T}\|p(t,\cdot)\|_{1,\cD}^{2}
    \leq C\big( \intl F \intl_{0,\cD}^{2} + E\|\phi\|_{1,\cD}^{2}\big),
  \end{equation}
  where the constant $C$ only depends on $K,\rho_{0},\kappa,T,$ and the function
  $\gamma$.
\end{thm}

\begin{rmk}
Since all constants $C$ in this paper are independent of $d_1$, all our
results in this paper may be extended to the more general equation
\eqref{eq:a1} where the $d_1$-dimensional standard Wiener process is replaced
with a cylindrical Wiener process.
\end{rmk}

\subsection{Solution in weighted Sobolev spaces and regularity}

Unfortunately, we could not establish any higher regularity for
BSPDEs to correspond to the theory of deterministic parabolic PDEs,
as given by Evans \cite[Theorem 7.1.6]{Evan98}, since the unknown
functions are not expected to be differentiable with respect to $t$
as in the deterministic sense. However, we can consider the
equations in weighted Sobolev spaces allowing the derivatives of the
solutions to blow up near the boundary, and furthermore obtain an
interior regularity for BSPDE \eqref{eq:a1} and \eqref{con:a1}.
\medskip

Let $\psi \in C^{2}_b(\overline{\mathcal{D}})$ be a nonnegative function such
that $\psi(x)=0$ for any $x\in \partial \cD$. Fix a positive integer $n$.

\begin{ass}\label{ass:c5}
  For any multi-index $\alpha$ such that $|\alpha|\leq n$, we have
  \begin{eqnarray*}
    &\psi^{|\alpha|}(|D^{\alpha}a^{ij}|+|D^{\alpha}b^i|+|D^{\alpha}c|
    +|D^{\alpha}\sigma^{ik}|+|D^{\alpha}\nu^k|)\leq K,\\
    &\psi^{|\alpha|}D^{\alpha}F \in \bH^{0}(\cD), \quad
    \psi^{|\alpha|}D^{\alpha}\phi\in L^{2}(\Omega, \sF_{T},
    H^{1}_{0}(\cD)).
  \end{eqnarray*}
\end{ass}

Note that Assumption \ref{ass:c5} implies the boundedness of the
functions $b,c$ and $\nu$, but does not imply Assumption
\ref{ass:c3}. We have the following

\begin{thm}\label{thm:c2}
  Let Assumptions \ref{ass:c1}, \ref{ass:c2}, \ref{ass:c3},
  and \ref{ass:c5} be satisfied. Then BSPDE \eqref{eq:a1} and \eqref{con:a1}
  has a unique strong solution $(p,q)$ such that for any multi-index $\alpha$ s.t.
  $|\alpha|\leq n$,
  \begin{eqnarray}\label{prp:c2}
    \begin{split}
      &(\psi^{|\alpha|}D^{\alpha}p,\psi^{|\alpha|}D^{\alpha}q)
      \in \cH^{2,1}_{0}(\cD),~~\\
      &\psi^{|\alpha|}D^{\alpha}p\in C([0,T],L^{2}(\cD))
      \cap L^{\infty}([0,T],H^{1}(\cD))~(a.s.),
    \end{split}
  \end{eqnarray}
  and moreover
  \begin{eqnarray}\label{leq:c2}
    \nonumber \intl \psi^{|\alpha|}D^{\alpha}p\intl_{2,\cD}^{2}
    +\intl \psi^{|\alpha|}D^{\alpha}q\intl_{1,\cD}^{2}
    + E\sup_{0\leq t\leq T}\|\psi^{|\alpha|}D^{\alpha}p(t,\cdot)\|_{1,\cD}^{2}\\
    \leq C \sum_{|\beta|\leq |\alpha|}\bigg( \intl \psi^{|\beta|}
    D^{\beta}F \intl_{0,\cD}^{2}
    + E\|\psi^{|\beta|}D^{\beta}\phi\|_{1,\cD}^{2}\bigg),
  \end{eqnarray}
  where the constant $C$ only depends on the norm of $\psi$ in
  $C^{2}(\overline{\cD})$, the parameters $K,\rho_{0},\kappa$, and $T$, and
  the function $\gamma$.
\end{thm}

We need the following lemma, which can be found in \cite{Kryl94}.

\begin{lem}\label{lem:c1}
  If both $v$  and $\psi Dv$ lie in $L^{2}(\mathcal{D})$, then $\psi
  v\in H^{1}_{0}(\mathcal{D})$.
\end{lem}

\begin{proof}
  Define $\mathcal{K}_{n} = \{x\in\cD: \textrm{dist}(x,\partial\mathcal{D})\geq
  4\cdot 2^{-n}\}.$ Take a nonnegative function $\zeta\in C^{\infty}_{0}(\bR^{d})$
  such that $\textrm{supp}\zeta \subset B_{1}(0),~\int_{\bR^{d}}\zeta = 1$.
  Define
  $\zeta_{n}(x)=2^{nd}\zeta(2^{n}x)$ and $\eta_{n}=\zeta_{n}*1_{\mathcal{K}_{n}}$.
  We have $\textrm{supp}(\eta_{n})\subset \mathcal{K}_{n+1}$ and
  $\eta_{n}|_{\mathcal{K}_{n-1}}=1$. Since $\psi\in C^{2}_b(\overline{\cD})$,
  we have $|\psi(x)|\leq C \textrm{dist}(x,\partial \cD).$
  It is not hard to show that $|\eta_{n}|\leq 1,~|\psi D\eta_{n}|\leq C,$ and
  $\eta_{n}\psi v\in H^{1}_{0}(\cD)$. Then we can get that both
  $\eta_{n}\psi v\rightarrow \psi v$ and $\eta_{n}D(\psi v)\rightarrow D(\psi v)$
  strongly in $L^{2}(\cD)$ as $n\rightarrow \infty$, and moreover
  \begin{eqnarray*}
    \begin{split}
      \int_{\cD}|D(\eta_{n}\psi v) - D(\psi v) |^{2}
      &\leq 2\int_{\cD}|\psi D\eta_{n}|^{2}|v|^{2}
      + 2\int_{\cD} |\eta_{n}D(\psi v) - D(\psi v) |^{2}\\
      &\leq C \int_{\cD\setminus\mathcal{K}_{n-1}}|v|^{2}
      + 2\int_{\cD} |\eta_{n}D(\psi v) - D(\psi v) |^{2}\\
      &\rightarrow 0 \quad \textrm{as } n\rightarrow \infty.
    \end{split}
  \end{eqnarray*}
  Hence $\psi v\in H^{1}_{0}(\mathcal{D})$.
\end{proof}

\begin{proof}[Proof of Theorem \ref{thm:c2}]
  The proof consists of two steps. We suppose for the moment
  that $\psi\in C^{n+2}(\overline{\cD})$,
  which will  finally be dispensed with.

  \emph{Step 1.} We first prove that Theorem  \ref{thm:c2} is true for the
  case where the leading coefficients $a$ and $\sigma$ are differentiable
  in the space variable $x$ with the gradients $a_{x}$ and
  $\sigma_{x}$ being bounded and thus equation \eqref{eq:a1} can be written
  into the divergence form \eqref{eq:a2}.

  We use the induction.
  Theorem \ref{thm:c1} shows that Theorem \ref{thm:c2} is true
  for the case of $n=0$.
  Assume that it is true for the case of $n=m-1~(m\geq 1)$.
  It is sufficient for us to show that it is true for the case of $n=m$.

  We assert that $(\psi^{m}D^{\alpha}p,\psi^{m}D^{\alpha}q)\in \bH^{1}_{0}(\cD)
  \times \bH^{0}(\cD;\bR^{d_1})$ for any multi-index $\alpha$ s.t. $|\alpha|=
  m$.
  Indeed, we know from our assumption that for any multi-index $\beta$ s.t.
  $|\beta|\leq m-1$,
  \[
    \psi^{m-1}D^{\beta}p\in \bH^{2}(\cD),~~\psi^{m-1}D^{\beta}q\in
    \bH^{1}(\cD;\bR^{d_1}).
  \]
  Keeping in mind that $\psi\in C^{2}_b(\overline{\cD})$,  we can
  easily get by induction that
  \begin{equation}\label{prp:c1}
    \psi^{m-1}D^{\beta}p_x \in \bH^{1}(\cD),
    \quad \psi^{m-1}D^{\beta}p_{xx}\in \bH^{0}(\cD),\quad \textrm{and }
    \psi^{m-1}D^{\beta}q_x\in \bH^{0}(\cD;\bR^{d_1}).
  \end{equation}
  Then we have
  \[
    \psi^{m-1}D^{\alpha}p
    \in \bH^{1}(\cD)\subset \bH^{0}(\cD),\quad
    \psi D(\psi^{m-1}D^{\alpha}p)
    \in \bH^{0}(\cD),\quad\textrm{and }
    \psi^{m}D^{\alpha}q\in \bH^{0}(\cD;\bR^{d_1}).
  \]
  In view of Lemma \ref{lem:c1}, we have from the first two relations that
  $\psi^{m}D^{\alpha}p\in \bH^{1}_{0}(\cD).$

  Take any $\eta\in C^{\infty}_{0}(\cD)$. Since $\psi\in
  C^{n+2}_b(\cD)$, we know that
  $D^{\alpha}(\psi^{m}\eta)\in H^{1}_{0}(\cD)$.
  From Proposition \ref{prop:b1}, we have
  \begin{eqnarray*}
    \begin{split}
      &\int_{\mathcal{D}}p(t,\cdot)D^{\alpha}(\psi^{m}\eta) dx
      - \int_{\mathcal{D}}\phi D^{\alpha}(\psi^{m}\eta) dx\\
      & = \int_{t}^{T}\int_{\mathcal{D}}
      \Big[ a^{ij}D_{ij}p+b^{i}D_{i}p - c p
      +\sigma^{ik}D_{i}q^{k}+\nu^k q^k
      +F\Big]D^{\alpha}(\psi^{m}\eta)dx dt\\
      &~~~ -\int_{t}^{T}\int_{\mathcal{D}}q^k D^{\alpha}(\psi^{m}\eta)
      dx dW^k_t,
      ~~~~\textrm{a.e.}~(\omega,t)\in \Omega\times [0,T].
    \end{split}
  \end{eqnarray*}
  Using the integration by parts formula, we show that the function pair
  $(\psi^{m}D^{\alpha}p,\psi^{m}D^{\alpha}q)\in \bH^{1}_{0}(\cD)
  \times \bH^{0}(\cD;\bR^{d_1})$ is a weak solution of the
  following BSPDE
  \begin{equation}\label{eq:c2}
    \left\{\begin{array}{l}
      d u = -\big[ a^{ij}D_{ij}u + \sigma^{ik}D_{i}v^k + \widetilde{F}\big]dt +
      v^k
      dW^k_{t},\\
      u(t,x)=0,~~~~~~ t\in [0,T],~~ x\in\partial \mathcal{D},\\
      u(T,x)=(\psi^{m}D^{\alpha}\phi)(x),~~~~ x\in \mathcal{D}
      \end{array}\right.
  \end{equation}
  with $u$ and $v$ being the unknown functions. Here
  \begin{eqnarray}\label{eq:c3}
    \begin{split}
      \widetilde{F}=& \sum_{\beta+\gamma=\alpha,|\beta|\geq 1}\Big[
      \big( \psi^{|\beta|}D^{\beta}a^{ij} \big)
      \big( \psi^{|\gamma|}D^{\gamma}p_{x^{i}x^{j}} \big) +
      \big( \psi^{|\beta|}D^{\beta}\sigma^{ik} \big)
      \big( \psi^{|\gamma|}D^{\gamma}q^k_{x^{i}} \big) \Big]\\
      &+ \sum_{\beta+\gamma=\alpha}\Big[
      \big( \psi^{|\beta|}D^{\beta}b^{i} \big)
      \big( \psi^{|\gamma|}D^{\gamma}p_{x^{i}} \big)
      -\big( \psi^{|\beta|}D^{\beta}c \big)
      \big( \psi^{|\gamma|}D^{\gamma}p \big)\Big]\\
      &- 2m a^{ij}\psi^{m-1} D_{i}\psi D^{\alpha}p_{x^j} -
      m a^{ij}\psi^{m-1} D_{ij}\psi D^{\alpha}p - m(m-1) a^{ij}\psi^{m-2}
      D_{i}\psi D_{j}\psi D^{\alpha}p\\
      & + \sum_{\beta+\gamma=\alpha} \big( \psi^{|\beta|}D^{\beta}\nu^k \big)
      \big( \psi^{|\gamma|}D^{\gamma}q^k \big)
      -m\sigma^{ik} \psi^{m-1}D_{i}\psi D^{\alpha}q^k
      + \psi^{m}D^{\alpha}F.
    \end{split}
  \end{eqnarray}
  From \eqref{prp:c1} and Assumption \ref{ass:c5}, we see that
  $\widetilde{F}\in \bH^{0}(\cD).$
  Moreover, from estimate \eqref{leq:c2} for $n=m-1$ (as a consequence
  of the induction assumption), we have
  \begin{eqnarray*}\begin{split}
    \intl \widetilde{F}\intl_{0,\cD}^{2}~
    \leq ~& C~ \bigg[ \sum_{|\beta|\leq m-1}\Big(\intl
    \psi^{|\beta|}D^{\beta}p\intl_{2,\cD}^{2}
    +\intl \psi^{|\beta|}D^{\beta}q\intl_{1,\cD}^{2} \Big)
    +\intl \psi^{m}D^{\alpha}F \intl_{0,\cD}^{2}\bigg]\\
    \leq ~& C~ \bigg[ \sum_{|\beta|\leq m-1}\Big( \intl \psi^{|\beta|}
    D^{\beta}F \intl_{0,\cD}^{2}
    + E\|\psi^{|\beta|}D^{\beta}\phi\|_{1,\cD}^{2}\Big)
    +\intl \psi^{m}D^{\alpha}F \intl_{0,\cD}^{2}\bigg],
  \end{split}\end{eqnarray*}
  where the constant $C$ only depends on the norm of $\psi$ in
  $C^{2}(\overline{\cD})$, the constants $K,\rho_{0},\kappa$ and $T$, and the function
  $\gamma$.
  Note that $\psi^{m}D^{\alpha}\phi\in \bH^{1}_{0}(\cD)$.
  Therefore, applying Theorem \ref{thm:c1} to BSPDE \eqref{eq:c2}, we
  have Theorem \ref{thm:c2} for $n=m$.
  \medskip

  \emph{Step 2.}  Now we remove the boundedness assumption
  on $a_{x}$ and $\sigma_x$ made in Step 1.

  In view of Theorem \ref{thm:c1}, BSPDE \eqref{eq:a1} and \eqref{con:a1}
  has a unique strong solution $(p,q)$.
  Due to Assumption \ref{ass:c3}, we can construct (e.g., by the
  standard technique of mollification) two sequences $a_{r}$ and $\sigma_r$
  with their first-order derivatives in $x$ being bounded, which converge uniformly (w.r.t.
  $(\omega,t,x)$) to $a$ and $\sigma$, respectively, such that all $a_{r}$ and $\sigma_{r}$
  satisfy  assumptions \ref{ass:c2}, \ref{ass:c3}, and \ref{ass:c5},
  with $\kappa$ in assumption~\ref{ass:c3} being replaced with $\kappa^{2}$.
  Then from Theorem \ref{thm:c1}, the following equation (for each $n$)
  \begin{equation*}\left\{\begin{array}{l}
    dp_{r}=-\big(a_{r}^{ij}D_{ij}p_{r}+b^{i}D_{i}p_{r}-cp_{r}
    +\sigma_{r}^{ik}D_{i}q_{r}^{k}+\nu^{k}q_{r}^{k}+F\big) dt
    +q_{r}^{k}dW_{t}^{k},\\
    p_{r}|_{x\in\ptl\cD}=0,~~~~~~
    p_{r}|_{t=T}=\phi \end{array}\right.
  \end{equation*}
  has a unique strong solution $(p_{r},q_{r})\in \cH^{2,1}_{0}(\cD)$, which
  satisfies estimate \eqref{leq:c1} with the constant $C$ being independent of $r$.
  Then we can check that $\{(p_{r},q_{r})\}$ is a Cauchy sequence in
  the space $\cH^{2,1}_{0}(\cD)$, whose limit is $(p,q)$. Similarly, we have
  that $\{(\psi^{|\alpha|}D^{\alpha}p_{r},\psi^{|\alpha|}D^{\alpha}q_{r})\}$
  ($|\alpha|\leq n$) is also a
  Cauchy sequence in the space $\cH^{2,1}_{0}(\cD)$, whose limit is
  denoted by $(u_{\alpha},v_{\alpha})$. Evidently, we have that
  $u_{\alpha}\in C([0,T],L^{2}(\cD))\cap L^{\infty}([0,T],H^{1}(\cD))$, and
  \begin{equation}\label{leq:c5}
    \| (u_{\alpha}, v_{\alpha})\|_{\cH^{2,1}(\cD)}^{2}
    + E\sup_{0\leq t\leq T}\|u_{\alpha}(t,\cdot)\|_{1,\cD}^{2}
    \leq C \sum_{|\beta|\leq |\alpha|}\bigg( \intl \psi^{|\beta|}
    D^{\beta}F \intl_{0,\cD}^{2}
    + E\|\psi^{|\beta|}D^{\beta}\phi\|_{1,\cD}^{2}\bigg).
  \end{equation}
  On the other hand, for every $\eta\in
  C_{0}^{\infty}(\cD)$ and a.e. $(\omega,t)$, we
  have ($|\alpha|\leq n$)
  \begin{equation*}
    \la \psi^{|\alpha|}D^{\alpha}p_{r}, \eta \ra
    = (-1)^{|\alpha|}\la p_{r},D^{\alpha}(\psi^{|\alpha|}\eta)\ra
    \rrow (-1)^{|\alpha|}\la p,D^{\alpha}(\psi^{|\alpha|}\eta)\ra
    =\la \psi^{|\alpha|}D^{\alpha}p, \eta \ra,
  \end{equation*}
  where we denote $\la u,v \ra=\int_{\cD}u(x)v(x)dx$. Thus
  $\psi^{|\alpha|}D^{\alpha}p=u_{\alpha}$. Similarly, we have
  $\psi^{|\alpha|}D^{\alpha}q=v_{\alpha}$. Estimate
  \eqref{leq:c2} is derived from inequality \eqref{leq:c5}.\medskip

  It remains to remove the additional assumption made
  at the beginning that $\psi\in
  C^{n+2}_b(\overline{\cD})$. Note that the constant $C$ in our estimate only depends
  on $|\psi|_{C^2}$ (and other parameters), which allows us to
  approximate $\psi$ in $C^{2}_b(\overline{\cD})$ by a sequence of nonnegative
  $C^{n+2}_b(\overline{\cD})$-functions vanishing on the boundary.
  Then in view of Lebesgue's dominated convergence theorem, the proof is complete.
\end{proof}

By choosing a proper weighting function $\psi$, we obtain the
following interior spacial regularity for the strong solution of
BSPDE \eqref{eq:a1} and \eqref{con:a1}.

\begin{cor}
  Let Assumptions \ref{ass:c1} and \ref{ass:c2} be satisfied.
  In addition, suppose that
  \begin{eqnarray}\label{con:c2}\begin{split}
    \sum_{|\alpha|\leq n}\big(|D^{\alpha}a^{ij}|+|D^{\alpha}b^i|+|D^{\alpha}c|
    +|D^{\alpha}\sigma^{ik}|+|D^{\alpha}\nu^k|\big)\leq K,\\
    F\in  \bH^{n}(\cD),~~~\phi\in L^{2}(\Omega,\sF_{T},H^{1}_{0}(\cD)\cap
    H^{n+1}(\cD)).~~~~
  \end{split}\end{eqnarray}
  Here the integer $n\geq 1$.
  Then BSPDE \eqref{eq:a1} and \eqref{con:a1} has a unique
  strong solution $(p,q)\in\cH^{2,1}_{0}(\cD)$ such that

    \emph{(i)} the functions $p$ and $q$ satisfy \eqref{prp:c1} and \eqref{leq:c1};

    \emph{(ii)} $p\in\bH^{n+2}_{loc}(\cD),
     ~q\in\bH^{n+1}_{loc}(\cD;\bR^{d_1}),
     ~p\in C([0,T],H^{n}_{loc}(\cD))(a.s.),$
     i.e., for any domain $\cD'\subset\subset\cD$, we have
     \begin{eqnarray}\label{leq:c3}
        \begin{split}
          &\intl p\intl_{n+2,\cD'}^{2}+\intl q\intl_{n+1,\cD'}^{2}
          + E\sup_{0\leq t\leq T}\|p(t,\cdot)\|_{n+1,\cD'}^{2}\\
          &~~~\leq C \sum_{m=-1}^{n}(\rho\wedge 1)^{-2(n-m)}
          \big( \intl F \intl_{m,\cD}^{2} + E\|\phi\|_{m+1,\cD}^{2}\big),
       \end{split}
    \end{eqnarray}
    with $\rho=\emph{dist}(\mathcal{D}',\partial\mathcal{D})$ and with the
    constant $C=C(K,\rho_{0},\kappa,T)$;

    \emph{(iii)} moreover, if $n-d/2>2$, then
    \begin{eqnarray*}
      &&p\in L^{2}(\Omega \times (0,T),\sP,C^{2}(\cD'))\cap
      L^{2}(\Omega, C([0,T]\times\cD')),\\
      &&q\in L^{2}(\Omega \times (0,T),\sP,C^{1}(\cD';\bR^{d_1})),
    \end{eqnarray*}
    for any domain $\cD'\subset\subset\cD$.
\end{cor}

\begin{proof}
  In view of Theorem \ref{thm:c1}, it remains to prove the assertions (ii)
  and (iii).

  Without loss of generality, we suppose $\rho\leq 1$.
  Define
  $$\mathcal{K} = \{x\in\cD: \textrm{dist}(x,\partial\mathcal{D})\geq
  \rho/2\}.$$
  It is clear that $\cD'\subset \mathcal{K}$.
  Take a nonnegative function $\zeta\in C^{\infty}_{0}(\bR^{d})$ such that
  \begin{eqnarray*}
    \begin{split}
      \textrm{supp}\zeta \subset B_{\frac{\rho}{2}}(0),~~~\int_{\bR^{d}}\zeta = 1,
      ~~~|D\zeta| \leq C \rho^{-1},~~~|D^2\zeta| \leq C \rho^{-2}.
    \end{split}
  \end{eqnarray*}
  Define $\psi = (\rho/2)\zeta * 1_{\mathcal{K}}$. It is not hard to show
  that $\psi$ is a well defined weight function for Theorem \ref{ass:c2}, and
  moreover
  \begin{equation}\label{con:c1}
    |\psi|\leq \frac{\rho}{2},~~~~
    |D\psi|\leq C,~~~~|D^2\psi|\leq C\rho^{-1},~~~~
    \psi\big|_{\cD}=\frac{\rho}{2}.
  \end{equation}

  In view of \eqref{eq:c3} and keeping in mind \eqref{con:c1},
  we show that for $n=1$,
  \begin{eqnarray*}
    \begin{split}
      &\intl \psi D p \intl_{2,\cD}^{2}
      +\intl \psi D q\intl_{1,\cD}^{2}
      + E\sup_{0\leq t\leq T}\|\psi D p(t,\cdot)\|_{1,\cD}^{2}\\
      &\leq~ C \big( \rho^{-2}\intl D_{x} p\intl_{0,\cD}^{2}
      +\intl p\intl_{2,\cD}^{2}
      +\intl q\intl_{1,\cD}^{2}
      + \intl \psi D F\intl_{0,\cD}^{2}
      +\intl  \psi D\phi \intl_{1,\cD}^{2}\big),
    \end{split}
  \end{eqnarray*}
  and for $2\leq m\leq n$ and
  multi-index $\alpha$ s.t. $|\alpha|=m$,
  \begin{eqnarray*}
    \begin{split}
      \intl \psi^{m} & D^{\alpha}p\intl_{2,\cD}^{2}
      +\intl \psi^{m}D^{\alpha}q\intl_{1,\cD}^{2}
      + E\sup_{0\leq t\leq T}\|\psi^{m}D^{\alpha}p(t,\cdot)\|_{1,\cD}^{2}\\
      \leq~& C\bigg[ \sum_{|\beta|\leq m-1}\big(\intl \psi^{|\beta|}D^{\beta} p
      \intl_{2,\cD}^{2} +\intl \psi^{|\beta|}D^{\beta} q
      \intl_{1,\cD}^{2}\big)\\
      &~~~+\intl \psi^{m}
      D^{\alpha}F \intl_{0,\cD}^{2}
      + E\|\psi^{m}D^{\alpha}\phi\|_{1,\cD}^{2}\bigg].
    \end{split}
  \end{eqnarray*}
  By induction, we have (here $|\alpha|=n$)
  \begin{eqnarray*}
    \begin{split}
      \intl \psi^{n} & D^{\alpha}p\intl_{2,\cD}^{2}
      +\intl \psi^{n}D^{\alpha}q\intl_{1,\cD}^{2}
      + E\sup_{0\leq t\leq T}\|\psi^{n}D^{\alpha}p(t,\cdot)\|_{1,\cD}^{2}\\
      \leq~& C  \rho^{-2}\intl D_{x} p\intl_{0,\cD}^{2}
      +C\sum_{|\beta|\leq n}\big(\intl \psi^{|\beta|}D^{\beta} F
      \intl_{0,\cD}^{2} +\intl \psi^{|\beta|}D^{\beta} \phi
      \intl_{1,\cD}^{2}\big).
    \end{split}
  \end{eqnarray*}
  By multiplying $\rho^{-2n}$ on both sides, we can easily get
  \eqref{leq:c3}. The assertion (iii) follows from the Sobolev embedding
  theorem. The proof is complete.
\end{proof}

\begin{rmk}
  In the case of $n-d/2>2$, the function pair $(p,q)$ satisfies equation
  \eqref{eq:b2} for all $(t,x)\in[0,T]\times\cD$ and $\omega\in \Omega'$ s.t.
  $\mathbb{P}(\Omega')=1$,
  which is a classical solution of equation \eqref{eq:a1}
  with the terminal-boundary condition \eqref{con:a1}
  (see e.g. \cite{MaYo99} for details).
\end{rmk}

\section{Proof of Theorem \ref{thm:c1}}

The proof is rather long, and it is divided into two subsections. We
first consider the simpler domain of a half space, and then go to
the general $C^2$ domain.

\subsection{The case of the half space}

In this subsection, we consider BSPDE \eqref{eq:a1} and
\eqref{con:a1} living in a half space.

Recall $\mathbb{R}^{d}_{+}=\{x\in \mathbb{R}^{d}:x^{1}>0\}$. Denote
$y=(x^{2},\dots,x^d)$.

\begin{defn}\label{defn:d1}
  We say a function $f$ defined on $\bR^d$ has the \emph{property of reflection
  invariance}, if $f(x^1,y)=-f(-x^1,y)$ for almost every $(x^1,y)\in
  \bR^d$.
\end{defn}

For a function $u$ defined on $\bR^{d}_+$, define $\widetilde{u}$ and
$\overline{u}$ as follows:
\begin{equation}\label{not:d1}
    \widetilde{u}=
    \left\{\begin{array}{ll}
      u,~&\textit{on}~~\{x^1>0\},\\
      0,~&\textit{on}~~\{x^1\leq 0\};
    \end{array}\right.~~~~~~~~~
    \overline{u}(x^1,y)=
    \left\{\begin{array}{ll}
      u(x^1,y),~&\textit{if}~~x^1>0,\\
      0,~& \textit{if}~~x^1=0,\\
      -u(-x^1,y),~&\textit{if}~~x^1 < 0.
    \end{array}\right.
\end{equation}
It is clear that $\overline{u}$ has the \emph{property of reflection
invariance}.

\begin{lem}\label{lem:d}
  \emph{(a)} Let $m$ be a positive number.
  Then a function $u\in H^{m}_{0}(\bR^{d}_+)$ if and only if $\widetilde{u}\in
  H^{m}(\bR^{d})$.

  \emph{(b)} The function $u\in H^{1}_{0}(\bR^d_+)$ if and only if $\overline{u}\in
  H^{1}(\bR^d)$.
\end{lem}

\begin{proof}
  The proof of assertion (a) can be found in, e.g., Chen \cite[Page 48]{Chen05}.
  The necessity of assertion (b) follows from assertion (a). It
  remains to prove the sufficiency of (b). Indeed, we can find $\vf_{n}\in
  C_{0}^{\infty}(\bR^d)$ such that $\vf_{n}\rrow \overline{u}$ in $H^{1}(\bR^d)$ as
  $n\rrow\infty$. Denote $\overline{\vf}_{n}(x^1,y)=-\vf_{n}(-x^1,y)$. Note that
  $\overline{u}(x^1,y)=-\overline{u}(-x^1,y)$. Thus we have
  $\overline{\vf}_{n}\rrow \overline{u}$ in
  $H^{1}(\bR^d)$. Define $\psi_{n}=(\vf_{n}+\overline{\vf}_{n})/2$.
  Then $\psi_{n}\rrow \overline{u}$ in $H^{1}(\bR^d)$. Since $\psi_{n}(0,y)=0$, the
  restriction of $\psi_{n}$ on $\bR^d_+$ belongs to $H^{1}_{0}(\bR^d_+)$.
  Thus $u\in H^{1}_{0}(\bR^d_+)$. The proof is complete.
\end{proof}

The following existence and uniqueness result concerning the Cauchy problem
of BSPDEs is taken from Du and Meng \cite[Prop. 4.1]{DuMe09}, which can also
be proved by means of the duality method of Zhou \cite{Zhou92} and Lemma
\ref{lem:b5}.

\begin{lem}\label{lem:d1}
  Consider the following BSPDE (on $\bR^d$)
  \begin{equation}\label{eq:d1}
    \left\{\begin{array}{l}
      dp=-[a^{ij}(t)D_{ij}p+\sigma^{ik}(t)D_{i}q^k+F]dt+q^kdW^k_t, \\
      p(T,x)=\phi(x),\ x\in \mathbb{R}^{d},
    \end{array}\right.
  \end{equation}
  where $a$ and $\sigma$ are two predictable processes taken values in $S^n$
  and $\bR^{d\times d_1}$, respectively, such that
  $\kappa I+ \sigma\sigma^{*}\leq 2a \leq \kappa^{-1}I,~\forall(\omega,t)$.
  Suppose $F\in \mathbb{H}^{0}(\bR^d),~\phi \in
  L^{2}(\Omega,\mathscr{F}_{T},H^{1}(\bR^d))$.
  Then BSPDE \eqref{eq:d1} has a unique strong
  solution $(p,q)$ in $\bH^{2}(\bR^d)\times\bH^{1}(\bR^d;\bR^{d_1})$
  such that $p\in C([0,T],H^{1}(\bR^d))~(a.s.)$, and moreover,
  \begin{equation}\label{leq:d2}
    \intl p \intl_{2,\bR^d}^{2} +\intl q \intl_{1,\bR^d}^{2}
    +E\sup_{0\leq t\leq T}\|p(t,\cdot)\|_{1,\bR^d}^{2}
    \leq C(\kappa,T)\big{(} \intl F \intl_{0,\bR^d}^{2} + E \|
    \phi\|_{1,\bR^d}^{2}\big{)}.
  \end{equation}
\end{lem}

Now we have the following

\begin{lem}\label{lem:e1}
  Let Assumptions \ref{ass:c2} be satisfied with
  $\cD=\bR^d_+$. Assume that $a$ and $\sigma$ are invariant in the space variable $x$.
  Suppose that $F\in \bH^{0}(\bR^{d}_{+})$ and
  $\phi\in L^{2}(\Omega,\sF_{T},H^{1}_{0}(\bR^{d}_{+})).$
  Then the following BSPDE
  \begin{equation}\label{eq:e1}
    \left\{\begin{array}{l}
      dp=-[a^{ij}(t)D_{ij}p+\sigma^{ik}(t)D_{i}q^k+F]dt+q^kdW^k_t, \\
      p(t,x)=0,\ x\in\partial \mathbb{R}^{d}_{+},\\
      p(T,x)=\phi(x),\ x\in \mathbb{R}^{d}_{+}
    \end{array}\right.
  \end{equation}
  has a unique strong
  solution $(p,q)\in \cH^{2,1}_{0}(\bR^{d}_{+})$ such that
  $p\in C([0,T],H^{1}(\bR^d_+))~(a.s.)$, and
  \begin{equation}\label{leq:e1}
    \|(p,q)\|_{\cH^{2,1}(\bR^{d}_{+})}^{2}
    + E\sup_{0\leq t\leq T}\|p(t,\cdot)\|_{1,\bR^{d}_{+}}^{2}
    \leq C\big( \intl F \intl_{0,\bR^{d}_{+}}^{2}
    + E\|\phi\|_{1,\bR^{d}_{+}}^{2}\big),
  \end{equation}
  where the constant $C$ depends only on $\kappa$ and $T$.
\end{lem}

\begin{proof}
  Recalling the definition \eqref{not:d1},
  we have $\overline{F}\in\bH^{0}(\bR^{d})$ and
  $\overline{\phi}\in L^{2}(\Omega,\sF_{T},H^{1}(\bR^{d}))$. From Lemma
  \ref{lem:d1}, the BSPDE
  \begin{equation}\label{eq:d2}
    \left\{\begin{array}{l}
      dP=-[a^{ij}(t)D_{ij}P+\sigma^{ik}(t)D_{i}Q^k+\overline{F}]dt+Q^kdW^k_t, \\
      P(T,x)=\overline{\phi}(x),\ x\in \mathbb{R}^{d}
    \end{array}\right.
  \end{equation}
  has a unique strong
  solution $(P,Q)$ in the
  space $\bH^{2}(\bR^d)\times\bH^{1}(\bR^d;\bR^{d_1})$
  such that $P\in C([0,T],H^{1}(\bR^d))~(a.s.)$, with the estimate
  \begin{equation}\label{leq:d3}
    \intl P \intl_{2,\bR^d}^{2} +\intl Q \intl_{1,\bR^d}^{2}
    +E\sup_{0\leq t\leq T}\|P(t,\cdot)\|_{1,\bR^d}^{2}
    \leq C(\kappa,T)\big{(} \intl \overline{F} \intl_{0,\bR^d}^{2} + E \|
    \overline{\phi}\|_{1,\bR^d}^{2}\big{)}.
  \end{equation}
  By symmetry and the uniqueness of the solution (of equation \eqref{eq:d2}),
  we know that $P$ and $Q$ have the
  \emph{property of reflection invariance}, for a.e.
  $(\omega,t)$. Denote by $p$ and $q$ the restrictions of $P$ and $Q$
  on $\bR^d_+$, respectively. From Lemma \ref{lem:d} (b), we know that
  $p\in \bH^1_0\cap\bH^{2}(\bR^d_+)$ and $q\in \bH^1_0(\bR^d_+;\bR^{d_1})$.
  Moreover,
  $p\in C([0,T],H^{1}(\bR^d_+))~(a.s.)$. It is evident that the pair $(p,q)$
  is a strong solution of equation \eqref{eq:e1}. Since every strong solution
  of equation \eqref{eq:e1}
  is also a weak solution, from the uniqueness of the weak solution,
  we know that $(p,q)$ is the unique strong solution of
  \eqref{eq:e1}. Estimate \eqref{leq:e1} follows from inequality
  \eqref{leq:d3}. The proof is complete.
\end{proof}

Now we prove the following perturbation result, which will be used
in the proof of Theorem \ref{thm:c1}.

\begin{prop}\label{prop:e}
  Consider the following BSPDE
  \begin{equation}\label{eq:e}
    \left\{\begin{array}{l}
      dp=-[a^{ij}D_{ij}p+\sigma^{ik}D_{i}q^k+F]dt+q^kdW^k_t, \\
      p(t,x)=0,\ x\in\partial \mathbb{R}^{d}_{+},\\
      p(T,x)=\phi(x),\ x\in \mathbb{R}^{d}_{+}.
    \end{array}\right.
  \end{equation}
  Assume that for a constant $\delta>0$ and for any $(\omega,t,x)$ we have
  \begin{equation}\label{con:e}
    |a(t,x)-a_{0}(t)|\leq\delta,\quad |\sigma(t,x)-\sigma_{0}(t)|\leq\delta,
  \end{equation}
  where $\{a_{0}(t): 0\leq t\leq T\}$ and
  $\{\sigma_{0}(t): 0\leq t\leq T\}$ are predicable processes
  satisfying Assumptions \ref{ass:c2}.
  Suppose that $F\in \bH^{0}(\bR^{d}_{+})$ and
  $\phi\in L^{2}(\Omega,\sF_{T},H^{1}_{0}(\bR^{d}_{+})).$

  Under the above assumptions, we assert that there exists a constant
  $\delta(\kappa,T)>0$ such that if $\delta\leq\delta(\kappa,T)$ then
  BSPDE \eqref{eq:e} has a unique strong solution
  $(p,q)\in \cH^{2,1}_{0}(\bR^{d}_{+})$ such that
  $p\in C([0,T],H^{1}(\bR^d_+))~(a.s.)$, and
  \begin{equation}\label{leq:e}
    \|(p,q)\|_{\cH^{2,1}(\bR^{d}_{+})}^{2}
    + E\sup_{0\leq t\leq T}\|p(t,\cdot)\|_{1,\bR^{d}_{+}}^{2}
    \leq C\big(\intl F \intl_{0,\bR^{d}_{+}}^{2}
    + E\|\phi\|_{1,\bR^{d}_{+}}^{2}\big),
  \end{equation}
  where the constant $C$ depends on $\kappa$ and $T$.
\end{prop}

\begin{proof}
  In view of Lemma \ref{lem:e1}, we know that for any $(u,v)\in
  \cH^{2,1}(\bR^{d}_{+})$,
  the BSPDE
  \begin{equation}\label{eq:e4}
    \left\{\begin{array}{l}
      dp =-\big{[} a_{0}^{ij}D_{ij}p+\sigma_{0}^{i}D_{i}q
      +(a^{ij}-a_{0}^{ij})D_{ij}u\\~~~~~~~~~~~~~~~~~~
      +(\sigma^{i}-\sigma_{0}^{i})D_{i}v +F \big{]}dt + qdW_{t},\\
      p(t,x)=0,~~~ x\in\partial \mathbb{R}^{d}_{+},\\
      p(T,x)=\phi(x),~~~ x\in \mathbb{R}^{d}_{+}
    \end{array}\right.
  \end{equation}
  has a unique strong solution
  $(p,q)\in \cH^{2,1}_{0}(\bR^{d}_{+})$ such that
  $p\in C([0,T],H^{1}(\bR^d_+))~(a.s.)$.
  We define the operator
  $A:\cH^{2,1}(\bR^{d}_{+}) \rightarrow \cH^{2,1}(\bR^{d}_{+})$
  as follows:
  $$ A(u,v) ~=~ (p,q). $$
  Then from estimate \eqref{leq:e1},
  we obtain that for
  $(u_{i},v_{i})\in \cH^{2,1}(\bR^{d}_{+}),~i=1,2$,
  \begin{eqnarray*}
    \|A(u_1-u_2,v_1-v_2)\|_{\cH^{2,1}(\bR^{d}_{+})}^{2}
    \leq C \delta
    \|(u_1-u_2,v_1-v_2)\|_{\cH^{2,1}(\bR^{d}_{+})}^{2}.
  \end{eqnarray*}
  Taking $\delta = (2C)^{-1}= (2C(\kappa,T))^{-1}$, we have that the
  operator $A$ is a contraction in $\cH^{2,1}(\bR^{d}_{+})$,
  and thus there exists a unique element $(p,q)\in \cH^{2,1}(\bR^{d}_{+})$
  such that $A(p,q)=(p,q)$. Furthermore, we have
  $(p,q)\in \cH^{2,1}_{0}(\bR^{d}_{+})$ and
  $p\in C([0,T],H^{1}(\bR^d_+))~(a.s.)$. It is clear that $(p,q)$
  is the unique strong solution
  of BSPDE \eqref{eq:e}.

  To establish estimate \eqref{leq:e},
  in view of Lemma \ref{lem:e1}, applying estimate \eqref{leq:e1} to
  equation \eqref{eq:e4}, we have
  \begin{equation*}
    \|(p,q)\|_{\cH^{2,1}(\bR^{d}_{+})}^{2}
    + E\sup_{0\leq t\leq T}\|p(t,\cdot)\|_{1,\bR^{d}_{+}}^{2}
    \leq C\delta \|(p,q)\|_{\cH^{2,1}(\bR^{d}_{+})}^{2}+
    C\big( \intl F \intl_{0,\bR^{d}_{+}}^{2}
    + E\|\phi\|_{1,\bR^{d}_{+}}^{2}\big).
  \end{equation*}
  Taking $\delta = (2C)^{-1}$, we obtain \eqref{leq:e}.
  The proof is complete.
\end{proof}

\subsection{The case of the general $C^2$ domain}

In this subsection, we shall complete the proof of Theorem \ref{thm:c1}. To
simplify the notation, we define
\begin{eqnarray*}
  \fH^{2,1}(\cD)~=&\big\{(u,v)\in \cH^{2,1}_{0}(\cD):u\in
    L^{2}(\Omega,L^{\infty}([0,T],H^{1}(\cD))),\\
    &~~~~u\in C([0,T],L^{2}(\cD))~(a.s.)\big\}
\end{eqnarray*}
being equipped with the norm
\begin{equation}
  \|(u,v)\|_{\fH^{2,1}(\cD)}=\big(\| (u, v)\|_{\cH^{2,1}(\cD)}^{2}
    + E\sup_{0\leq t\leq T}\|u(t,\cdot)\|_{1,\cD}^{2}\big)^{1/2}.
\end{equation}
It is clear that $\fH^{2,1}(\cD)$ is a Banach space.

First we have the following fact.

\begin{lem}\label{lem:f}
  Let $\Phi$ be the map defined in Assumption \ref{ass:c1} and $\Psi$ be the
  inverse of $\Phi$. Suppose $u\in H^{1}_{0}(\bR^{d}_{+})$ s.t.
  $\emph{supp}(u)\subset B_{+}\cup\partial\bR^{d}_{+}$. Then $u\circ\Psi\in
  H^{1}_{0}(\cD)$.
\end{lem}

\begin{proof}
  The proof is straightforward. Take $u_{n}\in C^{\infty}_{0}(B_+)$ such that
  $u_{n}\rightarrow u$ strongly in $H^{1}(\bR^{d}_{+})$ as
  $n\rightarrow\infty$. From the properties of $\Phi$, it is easy to show
  that $u_{n}\circ\Psi\in C^{2}(\cD)$ and $\textrm{supp}(u_{n}\circ\Psi)\subset U
  \cup \cD$ and then $u_{n}\circ\Psi\in H^{1}_{0}(\cD)$, where $U$ is taken
  from Assumption \ref{ass:c1}. Now we have
  \begin{eqnarray*}
    \begin{split}
      &\|u_{n}\circ\Psi - u\circ\Psi\|_{0,\cD}^{2} \leq |\det(D\Phi)|_{L^{\infty}}
      \|u_{n}-u\|_{0,B_+}^{2}\rightarrow 0,\\
      &\|D(u_{n}\circ\Psi -
      u\circ\Psi)\|_{0,\cD}^{2}\leq |\det(D\Phi)|_{L^{\infty}}
      |D\Psi|_{L^{\infty}}^2
      \|D(u_{n}-u)\|_{0,B_+}^{2}\rightarrow 0,
      \end{split}
  \end{eqnarray*}
  as $n\rightarrow\infty$, which implies $u\circ\Psi\in
  H^{1}_{0}(\cD)$.
\end{proof}

The following is Theorem \ref{thm:c1} under an additional assumption on the
coefficients $a$ and $\sigma$.

\begin{prop}\label{prop:f1}
  Let the conditions of Theorem \ref{thm:c1} be satisfied. In addition,
  assume that $a_{x}$ and $\sigma_{x}$ are bounded.
  Then BSPDE \eqref{eq:a1} and \eqref{con:a1}
  has a unique strong solution $(p,q)\in \fH^{2,1}(\cD)$ such that
  \begin{equation}\label{leq:g1}
    \intl p\intl_{2,\cD}^{2}+\intl q\intl_{1,\cD}^{2}
    + E\sup_{0\leq t\leq T}\|p(t,\cdot)\|_{1,\cD}^{2}
    \leq C\big( \intl F \intl_{0,\cD}^{2} + E\|\phi\|_{1,\cD}^{2}\big),
  \end{equation}
  where the constant $C$ only depends $K,\rho_{0},\kappa,T,$ and the function
  $\gamma$.
\end{prop}

\begin{proof}
Since $a_{x}$ and $\sigma_{x}$ are bounded, equation \eqref{eq:a1} can be
written into the divergence form
\begin{equation*}
    dp= -\big[ (a^{ij}p_{x^j}
    +\sigma^{ik}q^{k})_{x^i}+(b^{i}-a^{ij}_{x^j})p_{x^i}
    -c p + (\nu^{k}-\sigma^{ik}_{x^i})q^{k}+F\big]dt
    +q^{k}dW^{k}_{t}.
\end{equation*}
From Lemma \ref{lem:b3}, BSPDE \eqref{eq:a1} and \eqref{con:a1} has a unique
weak solution $(p,q)\in\bH^{1}_{0}(\cD)\times \bH^{0}(\cD;\bR^{d_1})$.

Now take a sufficiently small $\rho \in (0,\rho_{0})$ to satisfy the
following two conditions.
\begin{enumerate}
  \item[(1)] $\gamma(8\rho)\leq \delta$ with the constant
  $\delta=\delta(\kappa,T)$
  being given by Proposition \ref{prop:e}. In view of Assumption \ref{ass:c3},
  for any $(\omega,t)$ and $x,y\in \mathcal{D}$, we have
  \begin{equation}\label{prp:f}
    |a(t,x)-a(t,y)|\leq \delta,\quad |\sigma(t,x)-\sigma(t,y)|\leq \delta
  \end{equation}
  if $|x-y|\leq 8\rho$.

  \item[(2)] If $x,y$ belong to the same domain $U$ in Assumption
  \ref{ass:c1}, then for any $(\omega,t)$,
  \begin{eqnarray}\label{prp:f1}
    \nonumber &\big{|} D\Psi(x) a(t,x) (D\Psi(x))^{*}- D\Psi(y) a(t,y)
    (D\Psi(y))^{*}\big{|}\leq \delta_1, &\\
    &\big{|} D\Psi(x) \sigma(t,x) - D\Psi(y) \sigma(t,y) \big{|} \leq \delta_1&
  \end{eqnarray}
  if $|x-y|\leq 8\rho$, where the constant $\delta_1 =
  \delta(\kappa^2,T)$ and $\Psi$ is the inverse of $\Phi$.
\end{enumerate}

Then take a nonnegative function $\zeta\in C_{0}^{\infty}(\bR^{d})$ such that
$\textrm{supp}(\zeta)\subset B_{2\rho}(0)$, $\zeta(x)=1$ for $|x|\leq \rho$.
For any $z\in \mathbb{R}^d$, define for $x\in \bR^d$,
\begin{equation}\label{def:f1}
  \zeta^{z}(x)=\zeta(x-z),\quad p^{z}(t,x)=p(t,x)\zeta^{z}(x),\quad
  q^{z}(t,x)=q(t,x)\zeta^{z}(x).
\end{equation}
Then $(p^z,q^z)$ satisfies the following equation (in the sense of Definition
\ref{defn:b1} (ii))
\begin{eqnarray}\label{eq:g2}
  \begin{split}
  dp^z =& -\big[a^{ij}p^{z}_{x^i x^j} +\sigma^{ik}q^{z,k}_{x^i}+ \zeta^{z}F
  +(b^{i}\zeta^{z}-2a^{ij}\zeta^{z}_{x^j})p_{x^i}\\
  &~~~~~-(a^{ij}\zeta^{z}_{x^i x^j}
  +c)\zeta^{z}p
  +(\nu^{k}-\sigma^{ik}\zeta^{z}_{x^i})
  \zeta^{z}q^{k}\big]dt
  +q^{z,k} dW_{t}^{k}.
  \end{split}
\end{eqnarray}
In addition, define $$\eta^{z}(x)=\zeta(\frac{x-z}{2}).$$

Now consider the following two cases.
\medskip

\emph{Case 1.} $\textrm{dist}(z,\partial \mathcal{D})\leq 2\rho_0$. Then
$\mathcal{D}\cap \textrm{supp}(p^{z})\subset U \cap \mathcal{D}$, where $U$
is given in Assumption \ref{ass:c1}. Set $x=\Phi(y)$. Define
$$u^{z}(t,y)=p^{z}(t,x),\quad v^{z}(t,y)=q^{z}(t,x),
\quad (t,y) \in (0,T)\times \mathbb{R}^{d}_{y,+}.$$ Obviously
$(u^{z},v^{z})\in \mathbb{H}^{1}_{0}(\mathbb{R}^d_{y,+}) \times
\mathbb{H}^{0}(\mathbb{R}^d_{y,+};\bR^{d_1})$. Direct calculus shows that
\begin{eqnarray*}
  \begin{split}
    &p^{z}_{x^r x^s}(t,x)=\Psi^{i}_{x^r}(x)\Psi^{j}_{x^s}(x)u^{z}_{y^i y^j}(t,y)
    +(\zeta^{z}p_{x^i}+\zeta^{z}_{x^i}p)(t,x)\Psi^{i}_{x^r
    x^s}(x)\Phi^{r}_{y^i}(y)\\
    &q^{z,k}_{x^r}(t,x)=\Psi^{i}_{x^r}(x)v^{z,k}_{y^i}(t,y).
  \end{split}
\end{eqnarray*}
Substituting the above relations into equation \eqref{eq:g2}, it is not hard
to check that the functions $u^{z},v^{z}$ satisfy the BSPDE (in the sense of
Definition \ref{defn:b1} (ii))
\begin{equation}\label{eq:g1}
\left\{\begin{array}{ll}
  du^{z} = -\big(\tilde{a}^{ij}u^{z}_{y^iy^j}
  +\widetilde{\sigma}^{ik} v^{z,k}_{y^i}
  \widetilde{F} \big) dt + v^{z,k}dW_{t}^{k},\\
  u^{z}|_{\bR^{d}_{y,+}} = 0,~~~~u^{z}|_{t=T} = \zeta^{z}\phi,
\end{array}\right.
\end{equation}
where (observe that $u^{z}=0,v^{z}=0$ whenever $\eta^z \neq 1$)
\begin{eqnarray*}\begin{split}
& x=\Phi(y),\ x_0=\Phi(0),\ L^{0}=a^{rs}\partial^{2}_{x^r x^s},\\
& \widetilde{a}^{ij}(t,y) =
a^{rs}(t,x)\Psi^{i}_{x^r}\Psi^{j}_{x^s}(x)\eta^{z}(x)
+ a^{rs}(t,x_0)\Psi^{i}_{x^r}\Psi^{j}_{x^s}(x_0)(1-\eta^{z}(x)),\\
& \widetilde{\sigma}^{ik}(t,y) = \sigma^{rk}(t,x)\Psi^{i}_{x^r}(x)\eta^{z}(x)
+\sigma^{rk}(t,x_0)\Psi^{i}_{x^r}(x_0)(1-\eta^{z}(x)),\\
& \widetilde{F}(t,y) = (\zeta^{z}F)(t,x)+p_{x^r}(t,x)\Theta_{1}^{r}(t,y)
+p(t,x)\Theta_{2}(t,y)+q^{k}(t,x)\Theta_{3}^{k}(t,y),\\
& \Theta_{1}^{r}(t,y) = (\zeta^{z}L^{0}\Psi^{i})(t,x)\Phi^{r}_{y^i}(t,y)
+(b^{r}\zeta^{z}-2a^{rs}\zeta^{z}_{x^s})(t,x),\\
& \Theta_{2}(t,y) = (\zeta^{z}_{x^r}L^{0}\Psi^{i})(t,x)\Phi^{r}_{y^i}(t,y)
-(L^{0}\zeta^{z}+c\zeta^{z})(t,x),\\
& \Theta_{3}^{k}(t,y) = (\nu^{k}\zeta^{z}-\sigma^{rk}\zeta^{z}_{x^r}) (t,x).
\end{split}\end{eqnarray*}
In order to apply Proposition \ref{prop:e} to BSPDE \eqref{eq:g1}, we take
\begin{equation}
  a_{0}(t) = \widetilde{a}(t,0),~~~~\sigma_{0}(t) = \widetilde{\sigma}(t,0).
\end{equation}
Note that $\textrm{supp}(\eta^{z})\subset B_{4\rho}(z)$. Then it follows from
\eqref{prp:f1} that for any $y\in \bR^d_{y}$ and $x=\Phi(y)$, we have (recall
$x_{0}=\Phi(0)$)
\begin{eqnarray*}
  \begin{split}
    |\widetilde{a}(t,y)-\widetilde{a}(t,0)|&=|a^{rs}(t,x)\Psi^{i}_{x^r}\Psi^{j}_{x^s}(x)
    -a^{rs}(t,x_{0})\Psi^{i}_{x^r}\Psi^{j}_{x^s}(x_{0})|\cdot|\eta^{z}(x)|\\
    &\leq |a^{rs}(t,x)\Psi^{i}_{x^r}\Psi^{j}_{x^s}(x)
    -a^{rs}(t,x_{0})\Psi^{i}_{x^r}\Psi^{j}_{x^s}(x_{0})|\cdot 1_{B_{4\rho}(z)}\\
    &\leq \delta_{1}=\delta(\kappa^2,T).
  \end{split}
\end{eqnarray*}
Therefore, from Proposition \ref{prop:e}, BSPDE \eqref{eq:g1} has a unique
strong solution $(u,v)$ such that
$$ u \in\mathbb{H}^{1}_{0}\cap\bH^{2}(\mathbb{R}^{d}_{y,+}) , ~~~~v \in
\mathbb{H}^{1}_{0}(\mathbb{R}^{d}_{y,+};\bR^{d_1}).$$ It is clear that
$(u,v)$ is also a weak solution to BSPDE \eqref{eq:g1}. From the uniqueness
of the weak solution, we have $u=u^z$ and $v=v^z$. Hence we deduce that
\begin{equation}\label{prp:f2}
  u^{z} \in \mathbb{H}^{1}_{0}\cap\bH^{2}(\mathbb{R}^{d}_{y,+}),
  ~~~~v^{z} \in \mathbb{H}^{1}_{0}(\bR^{d}_{y,+};\bR^{d_1}).
\end{equation}
Denote $D(z,r)=B_{r}(z)\cap\cD$. Then \eqref{prp:f2} implies that restricted
on the domain $D(z,\rho)$, the solution
\begin{equation}\label{prp:f3}
  (p,q) \in \bH^{2}(D(z,\rho))\times\bH^{1}(D(z,\rho);\bR^{d_1})
\end{equation}
for any $z$ s.t. $\textrm{dist}(z,\partial \mathcal{D})\leq 2\rho_0$.

Now applying estimate \eqref{leq:e} to BSPDE \eqref{eq:g1}, we obtain that
\begin{eqnarray*}
  \intl u^z \intl_{2,\bR^{d}_{y,+}}^{2} + \intl v^z
  \intl_{1,\bR^{d}_{y,+}}^{2}
  +E\sup_{0\leq t\leq T}\|u^z(t,\cdot)\|_{1,\bR^{d}_{y,+}}^{2}~~~~~~~~~~~~~~~\\
  \leq C(\kappa,T) \big( \intl \widetilde{F} \intl_{0,\bR^{d}_{y,+}}^{2}+
  E\|(\zeta^{z}\phi)\circ\Phi\|_{1,\bR^{d}_{y,+}}^{2} \big).
\end{eqnarray*}
On the other hand, it is evident that (recall $\textrm{supp}(\zeta^z)\subset
B_{2\rho}(z)$)
\begin{eqnarray*}\begin{split}
  &\intl p \intl_{2,D(z,\rho)}^{2} + \intl q
  \intl_{1,D(z,\rho)}^{2}+E\sup_{0\leq t\leq T}\|p(t,\cdot)\|_{1,D(z,\rho)}^{2}\\
  &~~~~~\leq\intl p^z \intl_{2,\cD}^{2} + \intl q^z
  \intl_{1,\cD}^{2}+E\sup_{t\leq T}\|p^z(t,\cdot)\|_{1,\cD}^{2}\\
  &~~~~~\leq C\big(\intl u^z \intl_{2,\bR^{d}_{y,+}}^{2} + \intl v^z
  \intl_{1,\bR^{d}_{y,+}}^{2}
  +E\sup_{0\leq t\leq T}\|u^z(t,\cdot)\|_{1,\bR^{d}_{y,+}}^{2}\big),\\
  &\intl \widetilde{F} \intl_{0,\bR^{d}_{y,+}}^{2}
  \leq C\big(\intl F \intl_{0,D(z,2\rho)}^{2}
  + \intl p \intl_{1,D(z,2\rho)}^{2}
  +\intl q \intl_{0,D(z,2\rho)}^{2}\big),\\
  &E\|(\zeta^{z}\phi)\circ\Phi\|_{1,\bR^{d}_{y,+}}^{2}
  \leq C E\|\zeta^{z}\phi\|_{1,\cD}^{2}
  \leq C E\|\phi\|_{1,D(z,2\rho)}^{2},
  \end{split}
\end{eqnarray*}
where the constant $C$ depends only on $K$ and $\kappa$. Therefore, we obtain
\begin{eqnarray}\label{leq:f1}
  \begin{split}
    &\intl p \intl_{2,D(z,\rho)}^{2} + \intl q
    \intl_{1,D(z,\rho)}^{2}
    +E\sup_{0\leq t\leq T}\|p(t,\cdot)\|_{1,D(z,\rho)}^{2}\\
    &~\leq C\big( \intl F \intl_{0,D(z,2\rho)}^{2}
    +E\|\phi\|_{1,D(z,2\rho)}^{2} + \intl p \intl_{1,D(z,2\rho)}^{2}
    +\intl q \intl_{0,D(z,2\rho)}^{2}\big).
  \end{split}
\end{eqnarray}

\emph{Case 2.} $\textrm{dist}(z,\partial \mathcal{D})\geq 2\rho_0$. This case
can easily be reduced to the first one. Indeed, we can replace the domain
$\mathcal{D}$ by any half space with the boundary lying at a distance
$2\rho_0$ from $z$. In this situation it is not necessary to flatten the
boundary and to change coordinates. Then as above we deduce property
\eqref{prp:f3} for any $z\in \cD$ and obtain an estimate similar to
\eqref{leq:f1}.
\bigskip

Integrating both sides of inequality~\eqref{leq:f1} over $z\in
\mathbb{R}^{d}$, we obtain that
\begin{eqnarray}\label{leq:f2}
  \begin{split}
    &\intl p \intl_{2,\cD}^{2} + \intl q
    \intl_{1,\cD}^{2}+E\sup_{0\leq t\leq T}\|p(t,\cdot)\|_{1,\cD}^{2}\\
    &~\leq C\big( \intl F \intl_{0,\cD}^{2}
    +E\|\phi\|_{1,\cD}^{2} + \intl p \intl_{1,\cD}^{2}
    +\intl q \intl_{0,\cD}^{2}\big),
  \end{split}
\end{eqnarray}
where the constant $C$ depends on $K,\rho_{0},\kappa,T$, and the function
$\gamma$. Since $(p,q)\in \bH^1_0(\cD)\times\bH^0(\cD;\bR^{d_1})$, the
right-hand side is finite. Recalling that \eqref{prp:f3} holds true for any
$z\in \cD$, the above estimate implies that the unique weak solution of BSPDE
\eqref{eq:a1} and \eqref{con:a1} found by Lemma \ref{lem:b3} belongs to the
space $\cH^{2,1}(\cD)$, and moreover, $p\in C([0,T],L^{2}(\cD))\cap
L^{\infty}([0,T],H^{1}(\cD))~(a.s.)$. From Proposition \ref{prop:b1}, we know
that $(p,q)$ is the unique strong solution of BSPDE \eqref{eq:a1} and
\eqref{con:a1}.
\medskip

Replace the initial time zero by any $s\in[0,T)$. Proceeding identically as
before, we can obtain the following estimate similar to \eqref{leq:f2}
\begin{eqnarray}\label{leq:fs}
  \begin{split}
    &E \int_{s}^{T} \| p(t,\cdot) \|_{2,\cD}^{2}dt
    + E \int_{s}^{T} \| q(t,\cdot) \|_{1,\cD}^{2}dt
    +E\sup_{s\leq t\leq T}\|p(t,\cdot)\|_{1,\cD}^{2}\\
    &\leq C\left( \intl F \intl_{0,\cD}^{2}
    +E\|\phi\|_{1,\cD}^{2} + E \int_{s}^{T} \| p(t,\cdot) \|_{1,\cD}^{2}dt
    +E \int_{s}^{T} \| q(t,\cdot) \|_{0,\cD}^{2}dt \right).
  \end{split}
\end{eqnarray}

In view of the definition of the strong solution (Definition \ref{defn:b1}),
we know that the process $p(t,\cdot)$ is an $L^2(\cD)$-valued
semi-martingale. Then applying It\^o's formula for Hilbert-valued
semi-martingales (see e.g. \cite[Page 105]{PrZa92}), we have
\begin{eqnarray*}
    \begin{split}
      \|p(s,\cdot)\|_{0,\cD}^{2} = &\|\phi\|_{0,\cD}^{2}+2\int_{s}^{T}\int_{\cD}
      p \big(a^{ij}p_{x^i x^j}+b^{i}p_{x^i}-c p
      +\sigma^{ik}q^{k}_{x^i}+ \nu^{k}q^{k}+F\big) dx dt\\
      &~~~~ - \int_{s}^{T}\|q(t,\cdot)\|_{0,\cD}^{2}dt
      -2\int_{s}^{T}\int_{\cD}p q^k
      dx dW^k_t.
    \end{split}
\end{eqnarray*}
Taking expectations and using the Cauchy-Schwarz inequality, we have for any
$\eps>0$
\begin{eqnarray*}
    \begin{split}
      E \int_{s}^{T} \| q(t,\cdot) \|_{0,\cD}^{2}
       ~\leq ~ & E\|\phi\|_{0,\cD}^{2}+
      2 E \int_{s}^{T}\int_{\cD}
      p \big(a^{ij}p_{x^i x^j}+b^{i}p_{x^i}-c p
      +\sigma^{ik}q^{k}_{x^i}+ \nu^{k}q^{k}+F\big) dx dt\\
      \leq~ &E\|\phi\|_{0,\cD}^{2}
      + \eps E \int_{s}^{T}\big(\| p(t,\cdot) \|_{2,\cD}^{2}
      +\| q(t,\cdot) \|_{1,\cD}^{2}\big)dt \\
      &+ C(\eps,K) E \int_{s}^{T}\| p(t,\cdot) \|_{0,\cD}^{2}dt
      + \intl F \intl_{0,\cD}^{2}.
    \end{split}
\end{eqnarray*}
Taking $\eps$ small enough and recalling \eqref{leq:fs}, we have
\begin{eqnarray}\label{leq:g2}
  \begin{split}
    E \int_{s}^{T} \| p(t,\cdot) \|_{2,\cD}^{2}dt
    + E \int_{s}^{T} \| q(t,\cdot) \|_{1,\cD}^{2}dt
    +E\sup_{s\leq t\leq T}\|p(t,\cdot)\|_{1,\cD}^{2}\\
    \leq C\big( \intl F \intl_{0,\cD}^{2}
    +E\|\phi\|_{1,\cD}^{2}
    + E \int_{s}^{T} \| p(t,\cdot) \|_{1,\cD}^{2}dt\big),
  \end{split}
\end{eqnarray}
where the constant $C$ depends on $K,\rho_{0},\kappa,T$, and the function
$\gamma$. In particular, we have
\begin{equation*}
E \|p(s,\cdot)\|_{1,\cD}^{2} \leq C \bigg(\intl F \intl_{0,\cD}^{2}+E
\|\phi\|_{1,\cD}^{2} +E\int_{s}^{T}\| p(t,\cdot) \|_{1,\cD}^{2}dt\bigg).
\end{equation*}
Using the Gronwall inequality, we have
\begin{equation*}
\intl p\intl_{1,\cD}^{2} = \int_{0}^{T}E \|p(s,\cdot)\|_{1,\cD}^{2} \leq  C
e^{CT} \big(\intl F \intl_{0,\cD}^{2}+E \|\phi\|_{1,\cD}^{2}\big).
\end{equation*}
The last inequality along with \eqref{leq:g2} yields estimate \eqref{leq:g1}.
\medskip

It remains to prove $q\in \bH^{1}_{0}(\cD;\bR^{d_1})$. Since $q\in
\bH^{1}(\cD;\bR^{d_1})$, it is sufficient to check $q^{z}\in
\bH^{1}_{0}(\cD;\bR^{d_1})$ for each $z\in\partial\cD$ (recall
\eqref{def:f1}). Since $ v^{z}=q^{z}\circ \Phi
 \in \bH^{1}_{0}(\bR^{d};\bR^{d_1})$, by virtue of Lemma \ref{lem:f}, we get
$q^{z}\in \bH^{1}_{0}(\cD;\bR^{d_1})$. The proof is complete.
\end{proof}

\begin{rmk}
  The constant $C$ appearing in estimate
  \eqref{leq:g1} does not depend on $|a_{x}|$ and $|\sigma_x|$.
\end{rmk}

Proceeding identically to the proof of Proposition \ref{prop:f1}, we can
prove the following

\begin{prop}\label{prop:f2}
  Let the conditions of Theorem \ref{thm:c1} be satisfied. In addition,
  assume that the function pair $(p,q)\in \cH^{2,1}(\cD)$ is a strong
  solution of BSPDE \eqref{eq:a1} and \eqref{con:a1}.
  Then $(p,q)\in\fH^{2,1}(\cD)$,
  and there exists a constant $C$ only
  depending on $K,\rho_0,\kappa,T$ and the function $\gamma$ such that
  \begin{equation}\label{leq:g3}
    \|(p,q)\|_{\fH^{2,1}(\cD)}^{2}
    \leq C\big{(} \intl F \intl_{0,\cD}^{2} + E \|
    \phi\|_{1,\cD}^{2}\big{)}.
  \end{equation}
\end{prop}

Now we use the standard method of continuation to prove Theorem
\ref{thm:c1}.

\begin{proof}[Proof of Theorem \ref{thm:c1}]
  The uniqueness of the strong solution to equation \eqref{eq:a1} is
  an immediate consequence of estimate~\eqref{leq:g3}.
  Now we define
  \begin{eqnarray*}
    \begin{split}
      &\cL_{0} = a^{ij}(t,0)D_{ij}+b^{i}(t,x)D_{i}-c(t,x),~~~
      &\cM^{k}_{0} = \sigma^{ik}(t,0)D_{i}+\nu^{k}(t,x),\\
      &\cL_{1} = a^{ij}(t,x)D_{ij}+b^{i}(t,x)D_{i}-c(t,x),~~~
      &\cM^{k}_{1} = \sigma^{ik}(t,x)D_{i}+\nu^{k}(t,x).
    \end{split}
  \end{eqnarray*}
  For $\lambda\in[0,1]$, define
  \begin{equation*}
    \cL_{\lambda} = (1-\lambda)\cL_{0} + \lambda\cL_{1} ,~~~
    \cM^{k}_{\lambda} = (1-\lambda)\cM^{k}_{0} + \lambda\cM^{k}_{1}.
  \end{equation*}
  Consider the following equation
  \begin{equation}\label{eq:h1}
    dp=-(\cL_{\lambda}p+\cM^{k}_{\lambda}q^{k}+F)dt+q^{k}dW^{k}_{t},
    ~~~~~~p\big|_{x\in\ptl\cD}=0,~~p\big|_{t=T}=\phi.
  \end{equation}
  It is clear that the coefficients of equation \eqref{eq:h1}
  satisfy the conditions
  of Theorem \ref{thm:c1} with the same $K,\kappa$ and $\gamma$.
  Hence a priori estimate \eqref{leq:g3} holds
  for equation \eqref{eq:h1} for each $\lambda\in[0,1]$
  with the same constant $C$ (i.e., independent of
  $\lambda$).

  Assume that for some $\lambda=\lambda_{0}\in[0,1]$, equation \eqref{eq:h1} is
  solvable, i.e., it has
  a unique strong solution $(p,q)$ such that $(p,q)\in \fH^{2,1}(\cD)$
  for any $F\in \bH^0(\cD)$ and
  any $\phi\in L^{2}(\Omega,\sF_{T},H^{1}_{0}(\cD))$.
  For other $\lambda\in[0,1]$, we
  can rewrite \eqref{eq:h1} as
  \begin{eqnarray*}
    \begin{split}
      dp=-\big\{\cL_{\lambda_{0}}p+\cM^{k}_{\lambda_{0}}q^{k}
      +(\lambda-\lambda_{0})\big[(\cL_{1}-\cL_{0})p
      +(\cM^{k}_{1}-\cM^{k}_{0})q^{k}\big]
      +F\big\}dt+q^{k}dW^{k}_{t}.
    \end{split}
  \end{eqnarray*}
  Thus for any $(u,v)\in \cH^{2,1}(\cD)$, the equation
  \begin{eqnarray*}
    \begin{split}
      dp=-\big\{\cL_{\lambda_{0}}p+\cM^{k}_{\lambda_{0}}q^{k}
      +(\lambda-\lambda_{0})\big[(\cL_{1}-\cL_{0})u
      +(\cM^{k}_{1}-\cM^{k}_{0})v^{k}\big]
      +F\big\}dt+v^{k}dW^{k}_{t},
    \end{split}
  \end{eqnarray*}
  with the boundary conditions $p|_{t=T}=\phi$ and $p|_{x\in\ptl\cD}=0$,
  has a unique strong
  solution $(p,q)$ such that $(p,q)\in \fH^{2,1}(\cD)$.
  Then we define the operator
  $$A:~~\cH^{2,1}(\cD)~\rightarrow~\cH^{2,1}(\cD)$$
  as follows: $$A(u,v) = (p,q).$$
  Note that $A(u,v)\in \fH^{2,1}(\cD)$.
  Then from estimate \eqref{leq:g3},
  we can easily obtain that for any
  $(u_{i},v_{i})\in \cH^{2,1}(\cD),~i=1,2$,
  \begin{eqnarray}\label{leq:h1}
    \|A(u_1-u_2,v_1-v_2)\|_{\cH^{2,1}(\cD)}^{2}\leq
    \|A(u_1-u_2,v_1-v_2)\|_{\fH^{2,1}(\cD)}^{2}\\
    \leq C |\lambda-\lambda_{0}|
    \|(u_1-u_2,v_1-v_2)\|_{\cH^{2,1}(\cD)}^{2}.
  \end{eqnarray}
  Recall that the constant $C$ in \eqref{leq:h1} is independent of $\lambda$.
  Set $\theta = (2C)^{-1}$. Then the operator is contraction in
  $\cH^{2,1}(\cD)$ as long as $|\lambda-\lambda_{0}|\leq\theta$,
  which implies that equation \eqref{eq:h1} is solvable if
  $|\lambda-\lambda_{0}|\leq\theta$.

  Equation
  \eqref{eq:h1} is solvable for $\lambda=0$ in view of Proposition \ref{prop:f1}.
  Starting from $\lambda=0$, we get to $\lambda=1$ in finite steps, and
  this finishes the proof of solvability of equation \eqref{eq:a1}. The
  proof of Theorem \ref{thm:c1} is complete.
\end{proof}

\section{Some applications}

\subsection{A comparison theorem}

The comparison theorem plays an essential role in the theory of PDEs
and BSDEs. Ma and Yong \cite{MaYo99} gives some comparison theorems
for strong solutions to the Cauchy problem of degenerate BSPDEs by
It\^o's formula, which are improved by Du and Meng \cite{DuMe09}
under the super-parabolicity condition. In this subsection, we prove
the following comparison theorem for the strong solution to BSPDE
\eqref{eq:a1} and \eqref{con:a1} in the general $C^2$ domain.

\begin{thm}\label{thm:i1}
  Let the conditions of Theorem \ref{thm:c1} be satisfied,
  and $(p,q)$ be the unique strong solution to
  BSPDE \eqref{eq:a1} and \eqref{con:a1}.
  Suppose for any $(\omega,t)$,
  $F(t,\cdot)\geq 0$ and $\phi\geq 0$.
  Then $p(t,\cdot)\geq 0$ a.s., for every $t\in [0,T]$.
\end{thm}

The proof of Theorem \ref{thm:i1} needs the following lemma. In what follows,
we denote $a^- = -(a\wedge 0) $ for $a\in \bR$.

\begin{lem}\label{lem:i1}
  Let the conditions of Theorem \ref{thm:c1} be satisfied. In addition,
  assume that $a_{x}$ and $\sigma_{x}$ are bounded (by a constant $L$).
  Let $(p,q)$ be the strong solution of equation
  \eqref{eq:a2}. Then for some constant $C$,
  \begin{eqnarray}\label{leq:i1}
    \begin{split}
      E \int_{\cD}[p(t,x)^-]^{2}dx \leq e^{C(T-t)}\bigg\{
      E\int_{\cD} [\phi(x)^-]^{2}dx
      + E\int_{t}^{T}\int_{\cD}[F(s,x)^-]^{2}dxds\bigg\}.
    \end{split}
  \end{eqnarray}
\end{lem}

\begin{proof}
  Define the function $h:\bR\rightarrow [0,\infty)$ as follows:
  \begin{equation*}
    h(r) =
    \left\{\begin{array}{ll}
      r^2, & r\leq -1,\\
      (6r^3+8r^4+3r^5)^2, & -1 \leq r \leq 0,\\
      0, & r\geq 0.
    \end{array}\right.
  \end{equation*}
  Then $h$ is $C^2$ and
  $$h(0)=h'(0)=h''(0)=0,\ h(-1)=1,\ h'(-1)=-2,\ h''(-1)=2.$$
  For any $\eps > 0$, define $h_{\eps}(r)=\eps^{2}h(r/\eps)$. The function
  $h_{\eps}$ has the following properties:
  \begin{eqnarray*}
    \begin{split}
    &\lim_{\eps\rrow 0}h_{\eps}(r)=(r^-)^2,~~~
    \lim_{\eps\rrow 0}h_{\eps}'(r)=-2r^-,\quad \textrm{uniformly};\\
    &|h_{\eps}''(r)|\leq C,\quad \forall \eps > 0,r\in\bR;~~~~~
    \lim_{\eps\rrow 0}h_{\eps}''(r)= \left\{\begin{array}{ll}
    2,& r<0,\\
    0,& r>0.
    \end{array}\right.
    \end{split}
  \end{eqnarray*}

  Since $a_{x}$ and $\sigma_{x}$ exists and they are bounded, equation~\eqref{eq:a1} can
  be written into the divergence form. Then applying It\^o's formula for
  Hilbert-valued semi-martingales (see e.g. \cite[Page 105]{PrZa92}) to
  $h_{\eps}(p(t,\cdot))$, and from Green's formula,
  we obtain that
  \begin{eqnarray*}
    \begin{split}
      &E\int_{\cD}h_{\eps}(\phi(x))dx - E\int_{\cD}h_{\eps}(p(t,x))dx\\
      &=E\int_{t}^{T}\int_{\cD} \bigg\{ -h_{\eps}'(p)D_{i}(a^{ij}D_{j}p
      +\sigma^{ik}q^{k})-h_{\eps}'(p)\big[(b^{i}-D_{j}a^{ij})D_{i}p\\
      &~~~~~~~~-c p + (\nu^{k}-D_{i}\sigma^{ik})q^{k}+F\big]
      + \frac{1}{2}h_{\eps}''(p)|q|^{2}\bigg\}dx dt\\
      &=E\int_{t}^{T}\int_{\cD} \bigg\{ \frac{1}{2}h_{\eps}''(p)
      \big(2a^{ij}D_{i}pD_{j}p+2\sigma^{ik}q^{k}D_{i}p+|q|^{2}
      \big)\\
      &~~~~~~~~-h_{\eps}'(p)\big[(b^{i}-D_{j}a^{ij})D_{i}p
      -c p + (\nu^{k}-D_{i}\sigma^{ik})q^{k}+F\big]
      \bigg\}dx dt.
    \end{split}
  \end{eqnarray*}
  Setting $\eps\rrow 0$, by Lebesgue's Dominated Convergence Theorem, we
  have
  \begin{eqnarray*}
    \begin{split}
      &E\int_{\cD}[\phi(x)^{-}]^{2}dx - E\int_{\cD}[(p(t,x)^-]^{2}dx\\
      &=E\int_{t}^{T}\int_{\cD} 1_{\{p\leq 0\}}\cdot\big\{
      \big(2a^{ij}D_{i}pD_{j}p+2\sigma^{ik}q^{k}D_{i}p+|q|^{2}
      \big)\\
      &~~~~~~~~-2p\big[(b^{i}-D_{j}a^{ij})D_{i}p
      -c p + (\nu^{k}-D_{i}\sigma^{ik})q^{k}+F\big]
      \big\}dx dt.
    \end{split}
  \end{eqnarray*}
  For positive numbers $\delta,\delta_1$, we have
  \begin{eqnarray*}
    \begin{split}
      2a^{ij}D_{i}pD_{j}p+2\sigma^{ik}q^{k}D_{i}p+|q|^{2}
      &\geq 2a^{ij}D_{i}pD_{j}p-(1+\delta)|\sigma^{i}D_{i}p|^{2}
      +\frac{\delta}{1+\delta}|q|^{2}\\
      &\geq [-2\delta K + (1+\delta)\kappa]|D p|^{2}
      +\frac{\delta}{1+\delta}|q|^{2},\\
    \end{split}
  \end{eqnarray*}
  $$-p\big[(b^{i}-D_{j}a^{ij})D_{i}p
  -c p + (\nu^{k}-  D_{i}\sigma^{ik})q^{k}\big]
  \geq -\delta_{1}(|Dp|^{2}+|q|^{2}) - C(K,L)\delta_{1}^{-1}|p|^{2}.$$
  Taking $\delta$ and $\delta_{1}$ small enough (such that $\delta_1
  =\min\{-2\delta K + (1+\delta)\kappa,\frac{\delta}{1+\delta}\}>0$),
  we have
  \begin{eqnarray*}
    \begin{split}
      &E\int_{\cD}[\phi(x)^{-}]^{2}dx - E\int_{\cD}[(p(t,x)^-]^{2}dx\\
      &\geq E\int_{t}^{T}\int_{\cD} 1_{\{p\leq 0\}}\cdot\big[
      - C(\kappa,K,L)|p|^{2} - 2 p F
      \big]dx dt\\
      &\geq E\int_{t}^{T}\int_{\cD} \big[
      - C(\kappa,K,L)|p^{-}|^{2} - 2 p^{-} F^{-}
      \big]dx dt\\
      &\geq E\int_{t}^{T}\int_{\cD} \big[
      - C(\kappa,K,L)|p^{-}|^{2} - |F^{-}|^{2}
      \big]dx dt,
    \end{split}
  \end{eqnarray*}
  and this along with the Gronwall inequality implies the desired inequality
  \eqref{leq:i1}.
\end{proof}

\begin{proof}[Proof of Theorem \ref{thm:i1}]
  Due to Assumption \ref{ass:c3}, we can construct (e.g., by the
  standard technique of mollification) two sequences $a_{n}$ and $\sigma_n$
  with bounded first derivatives in $x$, which converge uniformly (w.r.t.
  $(\omega,t,x)$) to $a$ and $\sigma$, respectively, with $a_{n},\sigma_{n}$
  satisfying the same assumptions as $a,\sigma$ (with $\kappa^{2}$ instead of
  $\kappa$).
  Then, in view of Theorem \ref{thm:c1}, the following BSPDE (for each $n$)
  \begin{equation*}
  \left\{\begin{array}{l}
    dp_{n}=-\big(a_{n}^{ij}D_{ij}p_{n}+b^{i}D_{i}p_{n}-cp_{n}
    +\sigma_{n}^{ik}D_{i}q_{n}^{k}+\nu^{k}q_{n}^{k}+F\big) dt
    +q_{n}^{k}dW_{t}^{k},\\
    p_{n}|_{x\in\ptl\cD}=0,~~~~~~
    p_{n}|_{t=T}=\phi
  \end{array}\right.
  \end{equation*}
  has a unique strong solution $(p_{n},q_{n})\in \bH^{2}\times\bH^{1}$, such
  that
  \begin{equation}\label{leq:i2}
    \intl p_{n} \intl_{2}^{2} +\intl q_{n} \intl_{1}^{2}
    +E\sup_{0\leq t\leq T}\|p_{n}(t,\cdot)\|_{1}^{2}
    \leq C\big{(} \intl F \intl_{0}^{2} + E \|
    \phi\|_{1}^{2}\big{)},
  \end{equation}
  where the constant $C$ only depends on $K,\rho_0,\kappa,T$ and the function
  $\gamma$, and does not depend on $n$.
  It is easy to check that the function pair $(p-p_{n},q-q_{n})$ satisfies
  the following BSPDE
  \begin{equation*}
  \left\{\begin{array}{l}
    du=-\big(a^{ij}D_{ij}u+b^{i}D_{i}u-cu
    +\sigma^{ik}D_{i}v^{k}+\nu^{k}v^{k} + F_{n}\big) dt
    +v^{k}dW_{t}^{k},\\
    u|_{x\in\ptl\cD}=0,~~~~~~u|_{t=T}=0,
    \end{array}\right.
  \end{equation*}
  with $u$ and $v$ being the unknown functions, where
  \begin{equation*}
    F_{n}=(a^{ij}-a_{n}^{ij})D_{ij}p_{n}
    +(\sigma^{ik}-\sigma_{n}^{ik})D_{i}q_{n}^{k}.
  \end{equation*}
  In view of \eqref{leq:i2} and keeping in mind the uniform convergence of
  $a_{n}$ and $\sigma_{n}$, we have
  $$\intl F_{n} \intl_{0}\rrow 0\quad\textrm{as } n\rrow\infty,$$
  and this along with estimate \eqref{leq:c1} implies that for every
  $t\in[0,T]$,
  $$p_{n}(t,\cdot)\rrow p(t,\cdot),~~~~\textrm{strongly in}~
  L^{2}(\Omega\times\cD).$$
  On the other hand, it follows from Lemma \ref{lem:i1} that
  $p_{n}(t,\cdot)\geq 0$ a.s. for every $t\in[0,T]$. Hence we
  get $p(t,\cdot)\geq 0$ a.s. for every $t\in[0,T]$.
  The proof is complete.
\end{proof}

\subsection{Semi-linear equations in $C^2$ domains}

Consider the following BSPDE:
\begin{equation}\label{eq:j1}
  \left\{\begin{array}{l}
  \begin{split}
    dp(t,x)= & - \big[ a^{ij}(t,x)D_{ij}p(t,x) +
    \sigma^{ik}(t,x)D_{i}q^{k}(t,x)+ F(t,x,p(t,x),q(t,x))\big]dt\\
    & + q^{k}(t,x)dW_{t}^{k},
    \quad (t,x)\in [0,T)\times \cD;
  \end{split}\\
  \begin{split}
    &p(t,x)=0,~~~~t\in [0,T],\ x\in\partial \mathcal{D};\\
    &p(T,x)=\phi(x),~~~~x\in \mathcal{D}.
  \end{split}
  \end{array}\right.
\end{equation}
Such a BSPDE is associated to a FBSDE in a similar way as shown in Tang
\cite{Tang05}.

\begin{ass}\label{ass:j1}
  For any $(u,v)\in (H^{1}_{0}(\cD)\cap H^{2}(\cD))\times
  H^{1}(\cD;\bR^{d_1})$,
  the function $F(t,x,u,v)$
  is predictable as a function taking values in $H^{0}(\cD)$.
  The function $F$ is continuous in $(u,v)$. Moreover,
  for any $\eps>0$, there exists a constant $K_{\eps}$ such that, for any
  $(u_{i},v_{i})\in (H^{1}_{0}(\cD)\cap H^{2}(\cD))
  \times H^{1}(\cD;\bR^{d_1}),~i=1,2$
  and any $(\omega,t)$, we
  have
  \begin{eqnarray}\label{con:j1}
    \nonumber \|F(t,\cdot,u_{1},v_{1})-F(t,\cdot,u_{2},v_{2})\|_{0,\cD}
    \leq \eps\big(\|u_{1}-u_2\|_{2,\cD}+\|v_1-v_2\|_{1,\cD}\big)~~~~~~\\
    +K_\eps\big(\|u_1-u_2\|_{0,\cD}+\|v_1-v_2\|_{0,\cD}\big).
  \end{eqnarray}
\end{ass}

\begin{rmk}
  Take $F(t,x,u,v)=b^{i}D_{i}u-cu+\nu^{k}v^k$, and equation \eqref{eq:j1}
  becomes BSPDE \eqref{eq:a1} and \eqref{con:a1}.
  Under the boundedness of $b,c$ and
  $\nu$, we can easily check condition \eqref{con:j1} by the
  interpolation inequality. On the other hand,
  we shall see that the strong solution of BSPDE
  \eqref{eq:j1} belongs to $\cH^{2,1}_{0}(\cD)$ ($\subset\cH^{2,1}(\cD)$),
  which  allows us to discuss
  equation \eqref{eq:a1} with the coefficients $c,\nu$ blowing up near the
  boundary of $\cD$. However, we prefer not to purse this issue
  any further in this paper.
\end{rmk}

Then we have the following

\begin{thm}\label{thm:j1}
  Let the functions $a$ and $\sigma$ satisfy Assumptions
  $\ref{ass:c2}$ and $\ref{ass:c3}$. Let Assumptions \ref{ass:c1} and
  \ref{ass:j1} be satisfied. Suppose
  \begin{equation}\label{con:j2}
    F(\cdot,\cdot,0,0)\in\bH^{0}(\cD),~~~~~~
    \phi\in L^{2}(\Omega,\sF_{T},H^{1}_{0}(\cD)).
  \end{equation}
  Then BSPDE \eqref{eq:j1} has a unique strong solution
  $(p,q)\in \fH^{2,1}(\cD)$ such that
  \begin{equation}\label{leq:j2}
    \|(p,q)\|_{\fH^{2,1}(\cD)}^{2}
    \leq C\big{(} \intl F(\cdot,\cdot,0,0) \intl_{0,\cD}^{2} + E \|
    \phi\|_{1,\cD}^{2}\big{)},
  \end{equation}
  where the constant $C$ only depends on $K,\rho_{0},\kappa,T,$ the
  functions $\gamma$ and $K_{\eps}$.
\end{thm}

\begin{proof}
  \emph{Step 1.}
  First, we prove the a priori estimate \eqref{leq:j2} for
  the strong solution of equation \eqref{eq:j1}.

  From estimate \eqref{leq:c1}, there exists a constant $C$ depending
  only $K,\rho_{0},\kappa,T,$ and the function $\gamma$, such that
  \begin{equation*}
    \|(p,q)\|_{\fH^{2,1}(\cD)}^{2}
    \leq C\big{(} \intl F(\cdot,\cdot,p,q) \intl_{0,\cD}^{2} + E \|
    \phi\|_{1,\cD}^{2}\big{)}.
  \end{equation*}
  On the other hand, in view of Assumption \ref{ass:j1}, we have
  \begin{eqnarray}\label{leq:j6}\begin{split}
    &\intl F(\cdot,\cdot,p,q) \intl_{0,\cD}^{2}
    \leq 2\intl F(\cdot,\cdot,p,q)-F(\cdot,\cdot,0,0) \intl_{0,\cD}^{2}
    +2\intl F(\cdot,\cdot,0,0) \intl_{0,\cD}^{2}\\
    &~~\leq 4\eps^2\big(\intl p \intl_{2,\cD}^{2}+\intl q \intl_{1,\cD}^{2}\big)
    + 4K_{\eps}^2\big(\intl p \intl_{0,\cD}^{2}+\intl q \intl_{0,\cD}^{2}\big)
    + 2\intl F(\cdot,\cdot,0,0) \intl_{0,\cD}^{2}.
    \end{split}
  \end{eqnarray}
  Therefore, taking $\eps=(8C)^{-2}$, we have
  \begin{equation}\label{leq:j3}
    \|(p,q)\|_{\fH^{2,1}(\cD)}^{2}
    \leq C \big{(} \intl F(\cdot,\cdot,0,0) \intl_{0,\cD}^{2} + E \|
    \phi\|_{1,\cD}^{2}+
    \intl p \intl_{0,\cD}^{2}+\intl q \intl_{0,\cD}^{2}\big{)}.
  \end{equation}
  Proceeding similarly as in the proof of Proposition
  \ref{prop:f1}, we can remove the last two terms in the last inequality.
  Indeed, applying It\^o's formula for Hilbert-valued semi-martingales to
  $[p(t,\cdot)]^2$, we have
  \begin{eqnarray*}
    \begin{split}
      \|p(0,\cdot)\|_{0,\cD}^{2} = &\|\phi\|_{0,\cD}^{2}+2\int_{0}^{T}\int_{\cD}
      p \big(a^{ij}D_{ij}p+\sigma^{ik}D_{i}q^{k}+ F(t,x,p,q)\big) dx dt\\
      &~~~~ - \int_{0}^{T}\|q(t,\cdot)\|_{0,\cD}^{2}dt
      -2\int_{0}^{T}\int_{\cD}p q^k
      dx dW^k_t.
    \end{split}
  \end{eqnarray*}
  Taking expectations, using the Cauchy-Schwarz inequality, and keeping
  \eqref{leq:j6} in mind, we have
  for any $\delta>0$
  \begin{eqnarray*}
    \begin{split}
      \intl q \intl_{0,\cD}^{2} ~\leq ~ & E\|\phi\|_{0,\cD}^{2}+
      2 E \int_{0}^{T}\int_{\cD}
      p \big(a^{ij}D_{ij}p+\sigma^{ik}D_{i}q^{k}+ F(t,x,p,q)\big) dx dt\\
      \leq~ &E\|\phi\|_{0,\cD}^{2} + \delta\big(\intl p \intl_{2,\cD}^{2}
      +\intl q \intl_{1,\cD}^{2}
      + \intl F(\cdot,\cdot,p,q) \intl_{0,\cD}^{2}\big)
      + C(\delta,K)\intl p \intl_{0,\cD}^{2}\\
      \leq~&E\|\phi\|_{0,\cD}^{2} + \delta\big(\intl p \intl_{2,\cD}^{2}
      +\intl q \intl_{1,\cD}^{2}\big) + C(\delta,K)\big(
      \intl p \intl_{0,\cD}^{2}
      + \intl F(\cdot,\cdot,0,0) \intl_{0,\cD}^{2}\big).
    \end{split}
  \end{eqnarray*}
  Taking $\delta$ small enough and repeating \eqref{leq:j3}, we have
  \begin{equation*}
    \|(p,q)\|_{\fH^{2,1}(\cD)}^{2}
    \leq C \big{(} \intl F(\cdot,\cdot,0,0) \intl_{0,\cD}^{2} + E \|
    \phi\|_{1,\cD}^{2}+
    \intl p \intl_{0,\cD}^{2}\big{)},
  \end{equation*}
  where the constant $C$ only depends on $K,\rho_{0},\kappa,T,$ and the
  functions $\gamma$ and $K_{\eps}$.
  Using the Gronwall inequality, we obtain a priori estimate
  \eqref{leq:j2}.

  Furthermore, from a similar argument as above, we can prove
  the uniqueness of the strong solution of BSPDE \eqref{eq:j1}.\medskip

  \emph{Step 2.}
  We use the method of continuation to prove the solvability of BSPDE
  \eqref{eq:j1}. For each $\lambda\in[0,1]$, we consider the equation
  \begin{equation}\label{eq:j2}
    dp = -\big[ a^{ij}D_{ij}p+\sigma^{ik}D_{i}q+\lambda F(t,x,p,q)\big]dt
    +q^{k}dW_{t}^{k}, ~~~~p|_{x\in\ptl\cD}=0,~~p|_{t=T}=\phi.
  \end{equation}
  It is clear that the function $\lambda F$ satisfies Assumption \ref{ass:j1}
  with the same $K_\eps$ as $F$, and then equation \eqref{eq:j2} has a priori
  estimate \eqref{leq:j2} for each $\lambda$ with the same constant $C$.

  Assume that for $\lambda=\lambda_{0}\in[0,1]$, BSPDE \eqref{eq:j2} has
  a unique strong solution $(p,q)\in \fH^{2,1}(\cD)$,
  for any $\phi\in L^{2}(\Omega,\sF_{T},H^{1}_{0}(\cD))$ and
  any $F\in \bH^0(\cD)$ satisfying Assumption \ref{ass:j1} and condition
  \eqref{con:j2}.
  For other $\lambda\in[0,1]$, we
  can rewrite \eqref{eq:h1} as
  \begin{eqnarray*}
    \begin{split}
      dp=-\big\{a^{ij}D_{ij}p+\sigma^{ik}D_{i}q+\lambda_{0}
      F(t,x,p,q)+(\lambda-\lambda_{0})F(t,x,p,q)
      \big\}dt+q^{k}dW^{k}_{t}.
    \end{split}
  \end{eqnarray*}
  Thus for any $(u,v)\in \cH^{2,1}(\cD)$, the equation
  \begin{eqnarray*}
    \begin{split}
      dp=-\big\{a^{ij}D_{ij}p+\sigma^{ik}D_{i}q+\lambda_{0}
      F(t,x,p,q)+(\lambda-\lambda_{0})F(t,x,u ,v)\big\}dt+v^{k}dW^{k}_{t},
    \end{split}
  \end{eqnarray*}
  with the boundary conditions $p|_{t=T}=\phi$ and $p|_{x\in\ptl\cD}=0$,
  has a unique strong
  solution $(p,q)\in \fH^{2,1}(\cD)$.
  Then define the operator
  $$A:~~\cH^{2,1}(\cD)~\rightarrow~\cH^{2,1}(\cD)$$
  as $A(u,v) = (p,q).$ Note that $A(u,v)\in \fH^{2,1}(\cD)$.
  Proceeding similarly as in Step 1,
  we can easily obtain that for any
  $(u_{i},v_{i})\in \cH^{2,1}(\cD),~i=1,2$,
  \begin{eqnarray}\label{leq:j4}
    \|A(u_1-u_2,v_1-v_2)\|_{\cH^{2,1}(\cD)}^{2}\leq
    \|A(u_1-u_2,v_1-v_2)\|_{\fH^{2,1}(\cD)}^{2}\\
    \leq C |\lambda-\lambda_{0}|
    \|(u_1-u_2,v_1-v_2)\|_{\cH^{2,1}(\cD)}^{2}.
  \end{eqnarray}
  Recall that the constant $C$ in \eqref{leq:j4} does not depend on $\lambda$.
  Set $\theta = (2C)^{-1}$. Then the operator $A$ is a contraction in
  $\cH^{2,1}(\cD)$ as long as $|\lambda-\lambda_{0}|\leq\theta$,
  which implies that  \eqref{eq:j2} is solvable if
  $|\lambda-\lambda_{0}|\leq\theta$.

  The solvability of equation
  \eqref{eq:j1} for $\lambda=0$ has been given by Theorem \ref{thm:c1}.
  Starting from $\lambda=0$, we can reach $\lambda=1$ in finite steps, and
  this finishes the proof of solvability of equation \eqref{eq:j1}.
\end{proof}

\bibliographystyle{siam}

\end{document}